\documentclass[twoside,a4paper,10pt]{amsart}

\usepackage[symbol]{footmisc}
\usepackage{epigraph}
\usepackage{geometry}
\usepackage{amsmath}
\usepackage{amssymb}
\usepackage{amsthm}
\usepackage{mathrsfs}
\usepackage{mathtools}
\usepackage{esint}
\usepackage{graphicx}
\usepackage{color}
\usepackage{hyperref}
\usepackage{enumitem} 
\usepackage[english]{babel}
\usepackage[toc]{appendix}

\newcommand{\N}{\mathbb{N}}

\newcommand{\R}{\mathbb{R}}
\renewcommand{\S}{\mathbb{S}}

\newcommand{\cA}{\mathcal{A}}
\newcommand{\cB}{\mathcal{B}}
\newcommand{\cC}{\mathcal{C}}

\newcommand{\cF}{\mathcal{F}}

\newcommand{\cH}{\mathcal{H}}

\newcommand{\cL}{\mathcal{L}}

\newcommand{\cS}{\mathcal{S}}

\newcommand{\cW}{\mathcal{W}}

\newcommand{\cg}{\mathfrak{g}}

\newcommand{\bv}{\mathbf v}

\newcommand{\bR}{\mathbf{R}}

\newcommand{\Imm}{\mbox{\rm Im}\,}
\newcommand{\diam}{\mbox{\rm diam}}

\newcommand{\supp}{\mbox{\rm supp}}

\renewcommand{\div}{\mbox{\rm div}}

\newcommand{\bcup}{\bigcup}

\newcommand{\epty}{\emptyset}
\newcommand{\wto}{\rightharpoonup} 
\newcommand{\tens}{\otimes}
\newcommand{\sen}{\sin}

\newcommand{\res}{
	\,\raisebox{-.127ex}{\reflectbox{\rotatebox[origin=br]{-90}{$\lnot$}}}\,
} 

\newcommand{\pa}{\partial}
\newcommand{\con}{\subset}

\newcommand{\sm}{\setminus}
\newcommand{\lgl}{\langle}
\newcommand{\rgl}{\rangle}

\newcommand{\ep}{\varepsilon} 
\newcommand{\ga}{\gamma}
\newcommand{\be}{\beta}
\newcommand{\al}{\alpha}
\newcommand{\de}{\delta}

\newcommand{\la}{\lambda}

\newcommand{\om}{\omega}
\newcommand{\ro}{\rho}
\newcommand{\si}{\sigma}
\newcommand{\te}{\theta}
\newcommand{\De}{\Delta}
\newcommand{\Ga}{\Gamma}
\newcommand{\La}{\Lambda}
\newcommand{\Si}{\Sigma}

\newcommand{\Om}{\Omega}
\newcommand{\vp}{\varphi}

\newcommand{\sff}{\mbox{\rm II}}

\theoremstyle{plain}
\newtheorem{thm}{Theorem}[section] 

\theoremstyle{plain}

\theoremstyle{plain}
\newtheorem{prop}[thm]{Proposition}

\theoremstyle{plain}
\newtheorem{lemma}[thm]{Lemma}

\theoremstyle{plain}
\newtheorem{cor}[thm]{Corollary}

\theoremstyle{definition}

\theoremstyle{definition}
\newtheorem{remark}[thm]{Remark}

\theoremstyle{definition}

\title[The Plateau--Douglas Problem for the Willmore energy]{On the Plateau--Douglas problem for the Willmore energy of\\ surfaces with planar boundary curves}

\author{Marco Pozzetta}
\address{Dipartimento di Matematica, Universit\`{a} di Pisa, Largo Bruno Pontecorvo 5, 56127 Pisa, Italy}
\email{pozzetta@mail.dm.unipi.it}
\date{\today}

\begin{document}
	
\begin{abstract}
     For a smooth closed embedded planar curve $\Ga$, we consider the minimization problem of the Willmore energy among immersed surfaces of a given genus $\cg\ge1$ having the curve $\Ga$ as boundary, without any prescription on the conormal. In case $\Ga$ is a circle we prove that do not exist minimizers and that the infimum of the problem equals $\be_\cg-4\pi$, where $\be_\cg$ is the energy of the closed minimizing surface of genus $\cg$. We also prove that the same result also holds if $\Ga$ is a straight line for the suitable analogously defined minimization problem on asymptotically flat surfaces.\\
     Then we study the case in which $\Ga$ is compact, $\cg=1$ and the competitors are restricted to a suitable class $\cC$ of varifolds that includes embedded surfaces. We prove that under suitable assumptions minimizers exists in this class of generalized surfaces.
\end{abstract}

\maketitle

\noindent\textbf{MSC Codes (2010):} Primary: 49J40, 49J45, 49Q20, 53A05. Secondary: 49Q15, 53A30.

\vspace{0.4cm}

\section{Introduction}

\noindent In this work we consider immersed surfaces in $\R^3$ as follows. Fix an integer $\cg\ge 1$ and let $S_\cg$ be an abstract $2$-dimensional compact manifold of genus $\cg$. We call $\Si_\cg$ the $2$-dimensional manifold given by removing a topological disk with smooth boundary from $S_\cg$. We will consider smooth immersions $\Phi:\Si_\cg \to \R^3$ and we will usually call $\Si=\Phi(\Si_\cg)$ the immersed manifold. In this work a map is said to be smooth if it is of class $C^\infty$ up to the boundary.\\
In such a setting $\Si_\cg$ is endowed with the Riemannian metric $g_{ij}=\lgl \pa_i \Phi, \pa_j \Phi\rgl$ and area measure $d\mu_g$. For a local choice of unit normal vector $N$ on $\Si$, we define the vectorial second fundamental form as $\vec{\sff}(v,w)=-\lgl \pa_v N, w\rgl N$ and the scalar second fundamental form as $\sff(v,w)=-\lgl \pa_v N, w\rgl $, for any $v,w\in T_p\Si$. Therefore the mean curvature vector is $\vec{H}=\frac{1}{2}\sum_{i=1,2} \vec{\sff}(e_i,e_i)$ and the scalar mean curvature is $H=\frac{1}{2}\sum_{i=1,2} \sff(e_i,e_i)$, for any choice of an orthonormal basis $\{e_1,e_2\}$ of $T_p\Si$. We recall that equivalently in local chart one has $\sff_{ij}=\lgl N, \pa_i\pa_j \Phi \rgl$. We adopt the convention that in any product the repetition on upper and lower indexes means summation over those indexes. Then we can also write $H=\frac{1}{2}tr(\sff)=\frac{1}{2}g^{ij}\sff_{ij}$. Then we define the Willmore energy of $\Phi$ as
\begin{equation} \label{defwill}
\cW(\Phi):= \int_{\Si_\cg} |H|^2\,d\mu_g.
\end{equation}
The norm of the second fundamental form is $|\sff|=(g^{ia}g^{jb}\sff_{ij}\sff_{ab})^{\frac{1}{2}}$ and it will be useful the following quantity:
\begin{equation*}
D(\Phi):=\int_{\Si_\cg} |\sff|^2\,d\mu_g.
\end{equation*}
Assuming that $\Phi|_{\pa\Si_\cg}$ is an embedding, we also define
\begin{equation*}
G(\Phi)\equiv G(\Si):=\int_{\pa\Si} (k_g)_\Si\, d\cH^1,
\end{equation*}
with $(k_g)_\Si=\lgl \vec{k}_{\pa\Si},-co_\Si \rgl$ geodesic curvature of $\pa\Si$, where $\vec{k}_{\pa\Si}$ is the curvature vector of the curve $\pa\Si$, that is $\ddot{\ga}$ if $\ga$ parametrizes $\pa\Si$ by arc length, and $co_\Si$ is the unit outward conormal of $\Si$.\\

\noindent In the following we will need some definitions and results in the theory of curvature varifolds with boundary, for which we refer to Appendix A (see \cite{Ma} and \cite{Hu}). If $\Phi:S\to\R^3$ is smooth and proper and $S$ is some $2$-dimensional manifold, with a little abuse of notation we will use the symbol $\Si$ both to identify the curvature varifold with boundary induced by $\Phi$ and its support $\Phi(S)$; recall that is such case the varifold $\Imm(\Phi)$ induced by $\Phi$ is identified by the Radon measure $\mu_\Si=\te\cH^2\res(\Phi(S))$ where $\te(x)$ is the cardinality of the preimage $\Phi^{-1}(x)$.
Also, integration with respect to the measure $\mu_\Si$ induced in $\R^3$ by the varifold $\Si$ will be usually denoted by writing
\begin{equation*}
	\int f\,d\mu_\Si \equiv \int_\Si f,
\end{equation*}
for any $f\in C^0_c(\R^3)$. Therefore we will equivalently write the above energies as
\begin{equation*}
 \cW(\Phi)\equiv\cW(\Si)= \int_\Si |H|^2,
\end{equation*}
\begin{equation*}
D(\Phi)\equiv D(\Si)=\int_\Si |\sff|^2,
\end{equation*}
with $H,\sff$ generalized mean curvature and second fundamental form of $\Si$. We write $D(\Si\cap U)=\int_U |\sff|^2\,d\mu_\Si\equiv\int_{\Si\cap U} |\sff|^2$.\\
If $V$ is a curvature varifold with boundary, the symbol $\pa V$ will denote the boundary measure on the Grassmannian and $\si_V=\pi_\sharp(\pa V)$ will be the corresponding generalized boundary induced in $\R^3$.\\

\noindent We remark that the Willmore energy as defined in \eqref{defwill} is not conformally invariant because of the presence of a boundary, but it is invariant just under isometry and rescaling. However, recall that for surfaces with boundary the quantity $(\cW+G)$ is conformally invariant (\cite{Ch}).\\

\noindent Now consider a smooth embedded closed curve $\Ga\con\R^3$. In this paper we start the study of the following minimization problem
\begin{equation} \label{defpr}
\min\big\{\cW(\Phi)|\Phi:\Si_\cg\to\R^3 \mbox{ smooth immersion, } \Phi|_{\pa\Si_\cg}\to\Ga \mbox{ smooth embedding}  \big\},
\end{equation}
which we can call \emph{Plateau-Douglas for the Willmore energy}, since the constraints are just the boundary curve and the topology of the surface, as in the case of the classical Plateau-Douglas problem (\cite{DiHiTr3}). In particular we will mostly deal with planar boundary curves $\Ga$ and genus $\cg\ge1$. We will show that in this case the problem is nontrivial in the sense that not only there are no minimal surfaces among the competitors, but also the infimum of the problem is non zero, and this is ultimately due to the constraint on the genus.\\
Such a minimization problem is definitely spontaneous in the study of variational problems related to the Willmore functional and, in some sense, it is the direct analog with boundary of the problem proposed by Willmore himself about the minimization of $\cW$ among closed surface of a given genus (\cite{Wi65}), solved by putting together the results of \cite{SiEX} and \cite{BaKu}. From such classical problem we recall the following definitions:
\begin{equation*}
\begin{split}
	\be_\cg:=\inf\{ \cW(\Si):\Si\con\R^3 \mbox{ closed surface of genus } \cg \}=\min\{ \cW(\Si):\Si\con\R^3 \mbox{ closed surface of genus } \cg \},\\
\end{split}
\end{equation*}
\begin{equation*}
e_\cg:=\be_\cg-4\pi<4\pi \qquad\forall\cg,
\end{equation*}
which already play a role in the study of closed surfaces (see \cite{BaKu}).\\

\noindent The minimization problem of the Willmore energy for surfaces with boundary and the study of critical points is already present in the literature under two main formulations. The first is the presence of the clamped boundary condition, that is the additional constraint of having a prescribed smooth conormal field at the boundary. The latter is the problem with the so called Navier condition, that is the condition $H=0$ at the boundary, which arises naturally from the minimization problem without clamped condition (like \eqref{defpr}), and has already been studied mainly under the assumption that surfaces have rotational symmetry. Under this symmetry assumptions, recent results are contained in \cite{BeDaFr}, \cite{BeDaFr13}, \cite{DaDeGr}, \cite{DaDeWh}, \cite{DaFrGrSc}, \cite{DeGr}, \cite{Ei}, and \cite{EiGr}; a new result about symmetry breaking is \cite{Mandel}. We mention that an interesting problem about Willmore surfaces in a free boundary setting is considered in \cite{AlKu}. Other remarkable related results are achieved in \cite{Da12}, \cite{DeGrRo}, \cite{Pa}, and \cite{Ni93}.\\
\noindent Some works, and in particular \cite{SiEX}, \cite{KuSc}, and \cite{ScBP}, developed a very useful variational approach, that today goes under the name of \emph{Simon's ambient approach}. Such method relies on the measure theoretic notion of varifold as a generalization of the concept of immersed submanifold. One of the most relevant work for our purposes is \cite{ScBP}, in which the author adapts the method of \cite{SiEX} for proving the existence of branched immersions that are critical points of the Willmore energy under fixed clamped boundary conditions.\\
\noindent In this work we will deeply exploit first the conformal properties of the conformal Willmore functional $\cW+G$ recalled above, and then we will adopt the techniques of the varifold ambient approach of \cite{SiEX} and \cite{ScBP}, especially for what concerns regularity issues.
As other successful applications of the ambient approach we mention \cite{Schy}, in which the Willmore energy is minimized among closed surfaces with a constraint on the resulting isoperimetric ratio, and, more importantly in the setting of surfaces with boundary, we have the already cited \cite{ScBP}, and the more recent \cite{Ei19}, in which the author obtains results analogous to the ones in \cite{ScBP} for the Helfrich energy.\\
\noindent We remark that, more recently, an alternative and very powerful variational method based on a weak notion of immersions has been developed in \cite{Ri08}, \cite{RiLI}, and \cite{RiVP}. A recent application of these methods to the minimization of a Willmore-type energy amog disk type surfaces with clamped boundary data and constrained area is contained in \cite{DaPaRi}. Another recent application is \cite{MoSc}.\\

\noindent The numbers $e_\cg$ introduced above have a meaning in the study of minimization problems on asymptotically flat surfaces. In this paper we call asymptotically flat surface of genus $\cg$ without boundary with $K$ ends a complete orientable immersed $2$-dimensional manifold $\Phi:M\to\R^3$ such that:
\begin{enumerate}
	\item $M\simeq S_\cg\sm \sqcup_{i=1}^K \overline{D_i}$, i.e. $M$ is diffeomorphic to a genus $\cg$ surface with finitely many disjoint closed topological discs removed,
	\item for any $i=1,...,K$ there is $U_i$ open boundary chart at $D_i$ such that $U_i$ is diffeomorphic to an annulus with $\pa U_i=\pa D_i\sqcup \ga_i$ for a curve $\ga_i\simeq \S^1$, and there is an affine plane $\Pi_i$ such that for any $\ep>0$ there is $R>0$ such that $\Phi(U_i)\sm B_R(0)$ is the graph over $\Pi\sm K_R$ of a function $f_R$ with $\|f_R\|_{C^1}\le\ep$ where $K_R\con\Pi$ is compact,
	\item $D(\Phi)<+\infty$.
\end{enumerate}
\noindent If $\Phi$ defines an asymptotically flat surface of genus $\cg$ without boundary with $K$ ends as above, we call end of the surface one of the sets $\Phi(U_i)$.\\

\noindent In the following we will also consider asymptotically flat surfaces $\Si$ of genus $\cg$ with $K$ ends with boundary $\Ga$, meaning that $\Ga$ is a smooth complete embedding of $\R$ and $\Si\con\R^3$ is a subset such that the following properties hold.
\begin{enumerate}[label={\normalfont(\roman*)}]
	\item $\Si=\vp(\psi(M)\sm L)$ where $\psi$ is an embedding defining an asymptotically flat surface of genus $\cg$ without boundary with $K$ end, $\vp:\psi(M)\to\R^3$ is a complete immersion, $L$ is diffeomorphic to an open half-plane and it is contained in one end, say $E_1$, of $\psi(M)$. Moreover $\vp|_{\pa L}:\pa L\to \Ga$ is an embedding.
	
	\item For any $i=2,...,K$ for any end $E_i$ of $\psi$ there is an affine plane $\Pi_i$ such that for any $\ep>0$ there is $R>0$ such that $\vp(E_i)\sm B_R(0)$ is the graph over $\Pi_i\sm K_R$ of a function $f_R$ with $\|f_R\|_{C^1}\le\ep$ where $K_R\con\Pi_i$ is compact.
	
	\item There is an affine plane $\Pi_1$ such that for any $\ep>0$ there is $R>0$ such that $\vp(E_1\sm (B_R(0)\cup L))$ is the graph over $\Pi_1\sm H$ of a function $f_R$ with $\|f_R\|_{C^1}\le\ep$ where $H\con\Pi_1$ is smooth and diffeomorphic to a halfplane. 
	
	\item $D(\vp)<+\infty$.
\end{enumerate}

\begin{remark}
	One can verify that if $\Phi$ is an immersion of $S_\cg$ and $0\in\Phi(S_\cg)$, then $I\circ\Phi|_{S_\cg\sm\Phi^{-1}(0)}$ defines an asymptotically flat surface of genus $\cg$ without boundary, where $I(x)=\frac{x}{|x|^2}$. Conversely, if $\Phi$ defines an asymptotically flat surface of genus $\cg$ without boundary with one end, then $I\circ\Phi$ extends to a $C^{1,1}$ immersion of $S_\cg$; in particular $I\circ\Phi$ extends to an element of the class $\mathcal{E}_{S_\cg}$ defined in \cite{RiVP} (page 46).\\
	Therefore by Theorem 2.2 in \cite{BaKu} together with Theorem 1.7 in \cite{RiVP}, we have that the infimum of the Willmore energy among asymptotically flat surfaces of genus $\cg$ without boundary with one end is equal to $e_\cg$, and such infimum is achieved only by immersions of the form $I\circ\Phi|_{S_\cg\sm\Phi^{-1}(0)}$ for $\Phi:S_\cg\to\R^3$ embedding such that $0\in\Phi(S_\cg)$ and $\cW(\Phi)=\be_\cg$.
\end{remark}

\begin{remark}
	 For a fixed embedded closed planar curve $\Ga\con\R^2$ and a chosen $\cg$, we observe that the infimum of the corresponding problem \eqref{defpr} is $\le \be_\cg -4\pi$. Indeed we can consider an embedded asymptotically flat surface $\Si$ without boundary with one end and genus $\cg$ such that $\cW(\Si)=e_\cg$, that is, $\Si$ is minimizing among asymptotically flat surfaces of its own genus. Without loss of generality we can assume that the set $\{x^2+y^2\le1,z=0\}$ is strictly contained in the open planar region $\Om$ enclosed by $\Ga$. Chosen $\ep>0$, up to a translation, a rotation, and a rescaling of $\Si$ one can construct a competitor $\Si'$ for \eqref{defpr} such that $\Si'\cap \{x^2+y^2\le\tfrac12\}=\Si\cap  \{x^2+y^2\le\tfrac12 \}$, $\Si'\cap \{\tfrac12\le x^2+y^2\le1\}$ is the graph of a smooth function over the annulus $\{\tfrac12\le x^2+y^2\le 1\}$ with $\cW(\Si'\cap \{\tfrac12\le x^2+y^2\le1\})\le\ep$, and $\Si'\cap\{x^2+y^2\ge1\} =\Om \cap\{x^2+y^2\ge1\}$; hence $\cW(\Si')\le e_\cg + \ep$. For the explicit construction see the proof of Theorem \ref{thmmain}, where we will employ this argument several times.
\end{remark}

\noindent The first results we obtain are non-existence theorems for remarkable boundary curves, together with estimates for non-embedded surfaces. Here we sum up two results about this.

\begin{thm} \label{1}
	Let $\Ga$ be a circle and consider problem \eqref{defpr} for such curve. Then:
	\begin{enumerate}[label={\normalfont(\roman*)}]
		\item If $\pa\Si=\Ga$ and $\Si$ is not embedded, then $\cW(\Si)\ge4\pi$. In particular problem \eqref{defpr} can be restricted to embedded surfaces.
		\item If $\cg\ge1$, problem \eqref{defpr} has no solutions and the infimum equals $\be_\cg-4\pi=e_\cg$.
	\end{enumerate}
	The same statements hold for the analogous problem defined on genus $\cg$ asymptotically flat surfaces having a straight line as boundary.
\end{thm}

\begin{proof}
	The proof will follow by putting together Theorem \ref{nonexS1}, Corollary \ref{lowbound}, Corollary \ref{corequiv}, Corollary \ref{cors1equiv}, and Theorem \ref{thmr}.
\end{proof}

\noindent Then we consider the energy of surfaces having a fixed planar compact smooth closed curve $\Ga$ as boundary. In the following we say that a continuous proper map $\Phi:S\to\R^3$ is a \emph{branched immersion} if $\Phi|_{S\sm\{y_1,...,y_P\}}$ is a smooth immersion for some $y_1,...,y_P\in S$. For a such $\Phi$, its Willmore energy is defined by integrating over $S\sm\{y_1,..,y_P\}$.\\
For a fixed planar compact smooth closed curve $\Ga$, we will use Theorem \ref{1} in the study of problem
\begin{equation} \label{defpr2}
\min \big\{\cW(V)| V\in \cC(\Si_\cg) \big\},
\end{equation}
where
\begin{equation}\label{eq:Competitors}
\begin{split}
	\cC(\Si_\cg):=\big\{ V\,\,|\,\,&\exists\Si_n\in\cC^*(\Si_\cg):\,\Si_n\to V \mbox{ as varifolds, } D(\Si_n)\le C<+\infty, \cW(V)=\lim_n \cW(\Si_n),\, \si_V=\nu \cH^1\res \Ga,\\
	&V=\Imm(\Phi) \mbox{ with } \Phi:\Si_\cg\to\R^3 \mbox{ branched immersion} \big\},
\end{split}
\end{equation}
with
\begin{equation*}
\cC^*(\Si_\cg):=\big\{ \Si\con\R^3\,\,|\,\,\mbox{embedded surface of genus $\cg$, }\pa\Si=\Ga  \big\}.
\end{equation*}

\noindent Observe that $\cC(\Si_\cg)$ is a suitable subset of the varifold closure of embedded surfaces. The choice of such class of competitors is discussed in Remark \ref{remset} and it is essentially technical, although natural. We further define
\begin{equation*}
\begin{split}
\cC_{imm}(\Si_\cg):=\big\{ V\,\,|\,\,&\exists\Si_n\in\cC^*_{imm}(\Si_\cg):\,\Si_n\to V \mbox{ as varifolds, } D(\Si_n)\le C<+\infty, \cW(V)=\lim_n \cW(\Si_n),\, \si_V=\nu \cH^1\res \Ga,\\
&V=\Imm(\Phi) \mbox{ with } \Phi:\Si_\cg\to\R^3 \mbox{ branched immersion} \big\},
\end{split}
\end{equation*}
with
\begin{equation*}
\cC^*_{imm}(\Si_\cg):=\big\{ \Si\con\R^3\,\,|\,\,\mbox{immersed surface of genus $\cg$, }\pa\Si=\Ga  \big\}.
\end{equation*}

\noindent Here we sum up the other main results.

\begin{thm} \label{2}
	\textcolor{white}{text}
	\begin{enumerate}[label={\normalfont(\roman*)}]
		\item Let $\cg=1$. If
		\begin{equation*}
		\inf_{\cC(\Si_\cg)} \cW = \inf_{\cC_{imm}(\Si_\cg)} \cW  \,< e_1=\be_1-4\pi=2\pi^2-4\pi,
		\end{equation*}
		then problem \eqref{defpr2} has minimizers.
		\item Let $\cg=1$. There exist infinitely many closed convex planar smooth curves $\Ga$ such that if $\inf_{\cC(\Si_\cg)} \cW = \inf_{\cC_{imm}(\Si_\cg)} \cW$, then problem \eqref{defpr2} has minimizers.
	\end{enumerate}
%
\end{thm}

\begin{proof}
	The proof follows from Theorem \ref{thmmain}.\\
\end{proof}

\noindent We conclude this discussion with some observations.

\begin{remark}
	If the boundary curve is a circumference, the non-existence of minimizers persists also in the case of Problem \eqref{defpr2}. More precisely, we can say that if $\Ga=\S^1$ then Problem \eqref{defpr2} has no solution. Indeed we will see that by Lemma \ref{cop} and Corollary \ref{lowbound} a minimizing sequence of surfaces $\Si_n$ consists of embedded surfaces with conormal $co_n(p)=p$ for any $p\in \S^1$, i.e. the conormal of the sequence is fixed a posteriori. Assuming by contradiction that the sequence converges in the sense of varifolds to a minimizer $V\in\cC(\Si_\cg)$, the standard arguments of \cite{SiEX} and \cite{ScBP} imply that $V$ is actually a smooth embedded surface. This contradicts the non-existence result of Corollary \ref{cors1equiv}.\\
	It would be interesting to understand if $\S^1$ is the only compact convex curve such that the relative minimization problem has no minimizers. Taking into account also Theorem \ref{2}, in case of genus $1$ this would be equivalent to say that in the min-max problem
	\begin{equation*}
	\sup_{\Ga \mbox{ smooth, planar, convex}}\inf\big\{ \cW(\Si)|\pa\Si=\Ga,\,\Si\mbox{ embedded}\big\} =\be_1-4\pi
	\end{equation*}
	the supremum is only achieved by circumferences.\\
	We also mention that many boundary curves of truncated inverted Clifford-Willmore tori are convex and symmetric under two axes of reflection; it would be interesting to study if such symmetries are inherited by the minimizers of problem \eqref{defpr2}, and if the non-existence of minimizers with circular boundary is due to the rotational symmetry of the circumference, i.e. the presence of too many axes of reflection fixing the boundary.
\end{remark}

\begin{remark}
	An interesting open question arising from our results is to understand whether in point $i)$ of Theorem \ref{2} the ``if and only if'' holds, namely if the fact that the infimum equals $\be_1-4\pi$ implies the non existence of minimizers.
\end{remark}

\begin{remark}
	A natural extension of this work is the study of the same minimization problem for arbitrary non-planar embedded curves $\Ga\con\R^3$. In such cases the restriction of the problem to a family like \eqref{eq:Competitors} is no longer reasonable, since even minimal surfaces may be forced to have self-intersections. In this setting another fact to take into account is the existence of curves $\Ga$ solving the classical Plateau-Douglas problem for arbitrary fixed genus (see \cite{DiHiTr3}); the so called \emph{Douglas condition} on a curve $\Ga$ is sufficient for the solvability of the classical Plateau-Douglas, however no sufficient and necessary conditions are known, and this issue may be related to the minimization of the Willmore energy as formulated in this work.
\end{remark}

\subsection*{Organization}
\noindent The paper is organized as follows. In Section 2 we study problem \eqref{defpr} when $\Ga$ is a circumference, giving the first result of non-existence. In Section 3 we prove some lower bound estimates on non-embedded surfaces, completing the proof of Theorem \ref{1} in the case of the circular boundary. In Section 4 we mainly study the generalized problem \eqref{defpr2} applying Simon's ambient approach (\cite{SiEX}) and some techniques from \cite{ScBP}. Here we also study the minimization problem \eqref{defpr} suitably defined on asymptotically flat surfaces having a straight line as boundary, therefore completing the proof of Theorem \ref{1}; then we prove Theorem \ref{2} also using the results obtained in Section 1 and the results about such asymptotically flat surfaces as tools in the study of the general problem \eqref{defpr2}. Appendix A contains the definitions and the facts about varifold theory that are needed in this work. In Appendix B we collect some technical results from \cite{SiEX} for the convenience of the reader.

\textcolor{white}{text}


\section{$ \S^1 $ boundary datum}

\noindent Let
\begin{equation}\label{ref37} \S^1=\{ (x,y,z)\in\R^3|x^2+y^2=1,z=0 \}. \end{equation}
and let $D$ be the bounded planar region enclosed by $\S^1$. In this section we consider the following minimization problem.
\begin{equation} \label{problemS1}
\min \{\cW(\Phi)| \Phi:\Si_\cg\to\R^3 \mbox{ smooth immersion, } \Phi|_{\pa\Si_\cg}\to \S^1 \mbox{ smooth embedding} \}.
\end{equation}
Also, for a fixed $\cg$, let
$$ \cF:= \{\Phi:\Si_\cg\to\R^3 \mbox{ smooth immersion, } \Phi|_{\pa\Si_\cg}\to \S^1 \mbox{ smooth embedding} \}.$$

\noindent Let us introduce the notation:
\begin{equation}\label{eq1}
\begin{split}
	\forall r>0,c\in\R^3:\qquad &I_{r,c}:\R^3\sm\{c\}\to\R^3\sm\{0\},\quad I_{r,c}(p)=r^2\frac{p-c}{|p-c|^2},\\
	& I^{-1}_{r,c}:\R^3\sm\{0\}\to\R^3\sm\{c\},\quad I^{-1}_{r,c}(p)=c+r^2\frac{p}{|p|^2},
\end{split}
\end{equation}
Recall that by the classical Liouville's Theorem (\cite{Sp}), conformal maps on open sets of $\R^3$ are given by compositions of isometries, homotheties, and spherical inversions as defined in \eqref{eq1}.\\

\noindent It is useful to remember that, by \cite[Theorem 2.2]{BaKu}, if $\Phi:S_0\to\R^3$ is a smooth immersion of a closed surface and $c\in \Phi(S_0)$, and we define $\tilde\Phi: S_0\sm\Phi^{-1}(c)\to \R^3$ by $\tilde{\Phi}=I_{1,c}\circ \Phi$, then
\[
\cW(\tilde \Phi) = \cW(\Phi) - 4\pi \sharp\Phi^{-1}(c),
\]
where $\sharp\Phi^{-1}(c)$ denotes the cardinality of $\Phi^{-1}(c)$.\\

\noindent Without loss of generality we can assume that the orientation of the boundary is given by $t\mapsto (\cos t,\sin t,0)$, and the curvature vector of the boundary curve is just $\vec{k}(p)=-p$ for any $p\in \S^1$. Denoting by $co_\Phi$ the unit outward conormal, then $k_g(p)=\lgl \vec{k},-co_{\Phi}\rgl = \lgl p, co_{\Phi}(p)\rgl $ and
\begin{equation*}
\forall\Phi\in\cF: \qquad G(\Phi)=\int_{\S^1} \lgl p, co_{\Phi}(p)\rgl \, d\cH^1(p) .
\end{equation*}

\noindent Observe that if $\Phi\in\cF$, then $G(\Phi)\le2\pi$ with equality if and only if $co_\Phi(p)=p$ for any $p\in \S^1$.\\
We want to prove the following result.

\begin{thm} \label{nonexS1}
	For any genus $\cg\ge1$, problem \eqref{problemS1} has no minimizers and the infimum equals $\be_\cg-4\pi$.
\end{thm}

\noindent The proof is based in the following tool.

\begin{lemma}  \label{cop}
	Let $\Phi\in\cF$ and denote $\Si=\Phi(\Si_\cg)$. Then for any $\ep>0$ there is $F:U\to\R^3$ such that
	\begin{equation*}
	\begin{split}
		&i)\, U\con\R^3\mbox{ open, }\Si\con U,\\
		&ii)\, F:U\to F(U)\mbox{ conformal diffeomorphism,}\\
		&iii)\, \|co_{F\circ\Phi}(p)-p\|_{L^2(\S^1)} <\ep,\\
		&iv)\, 2\pi-G(F\circ\Phi)<\ep.
	\end{split}
	\end{equation*}
\end{lemma}

\begin{proof}
	If $G(\Phi)=2\pi$, then $co_\Phi(p)=p$ and $F=id|_{\R^3}$ works. So suppose in general that $G(\Phi)<2\pi$. Let $T_q$ and $D_\al$ be the maps
	\begin{equation*}
	\begin{split}
		&\forall\al>0:\quad D_\al(p)=\al p \qquad \forall p\in\R^3,\\
		&\forall q\in\R^3: \quad T_q(p)=p+q \qquad \forall p\in\R^3.
	\end{split}
	\end{equation*}
	Consider the point $(-1,0,0)=:v\in \S^1$ and the inversion $I_{1,v}$. Note that $I_{1,v}$ maps $\S^1\sm\{v\}$ onto the line $r_{-v/2}$ passing through the point $-v/2$, lying in the plane of $\S^1$ and parallel to $T_v \S^1$.\\
	Let $\bR$ be a rotation in $\R^3$ with axis $\{x=z=0\}$, we claim that the desired map $F$ is
	\begin{equation} \label{conf}
	F(p)=\begin{cases}
		I_{1,v}^{-1}\circ T_{-v/2}\circ \bR\circ D_\al \circ T_{v/2} \circ I_{1,v} (p) =   v+\frac{\al\bR\big[\frac{p-v}{|p-v|^2}+\frac{v}{2}\big]-\frac{v}{2}}{\big| \al\bR\big[\frac{p-v}{|p-v|^2}+\frac{v}{2}\big]-\frac{v}{2} \big|^2} & p\in U\sm\{ v\},\\
		v & p=v.
	\end{cases}
	\end{equation}
	for suitable choice of $\al\in(0,1)$ and rotation $\bR$, and $F$ is defined on
	\begin{equation} \label{defU}
	U=\R^3\sm \bigg\{I^{-1}_{1,v}\bigg(-\frac{1}{2}\bigg( \frac{1}{\al}\bR^{-1}[v]+v \bigg) \bigg)   \bigg\}.
	\end{equation}\\
	\noindent The surface $I_{1,v}(\Si)$ is an asymptotically flat manifold with $K$ ends, where $K\ge1$ is the multiplicity of $v$ in $\Si$. For $\be,\ga,\de\in(0,1)$ arbitrarily small there exist $\al=\al(\be,\ga)\in(0,1)$ sufficiently small and suitable $\bR=\bR(\de)$ so that
	\begin{equation} \label{ref-1}
	d_{C^1}\bigg(\bR\circ D_\al\circ T_{v/2} \circ I_{1,v} (\Si)\sm B_\ga(0), \bigcup_{i=1}^K \Pi_i  \sm B_\ga(0) \bigg)<\be,
	\end{equation}
	for a half plane $\Pi_1$ and planes $\Pi_2,...,\Pi_K$ passing through the origin with $\pa\Pi_1=\{x=y=0\}$, and
	\begin{equation} \label{ref0}
	\begin{split}
		& \lgl co_{\Pi_1},(-1,0,0)\rgl > 1-\de,\\
		&0\not\in T_{-v/2}\circ \bR\circ D_\al \circ T_{v/2} \circ I_{1,v} (\Si),
	\end{split}
	\end{equation}
	where $\,co_{\Pi_1}$ is the conormal vector of $\Pi_1$.\\
	Note that condition $0\not\in T_{-v/2}\circ \bR\circ D_\al \circ T_{v/2} \circ I_{1,v} (\Si)$ in \eqref{ref0} is equivalent to $I_{1,v}(p)\not=-\frac{1}{2}\big(\frac{1}{\al}\bR[v]+v \big)$ for any $p\in\Si$; and since $\frac{1}{\al}\bR[v]+v \neq 0$ for any $\al<1$, such condition is equivalent to $I^{-1}_{1,v}\big(-\frac{1}{2}\big(\frac{1}{\al}\bR[v]+v \big) \big)\not\in\Si$ which justifies the definition of $U$. So by \eqref{ref0} the function $F$ is well defined on $U$ and $\Si\con U$.\\
	Now for any $p\in U\sm\{v\}$ we have
	\begin{equation*}
	F(p)=v+ \frac{\al\bR[p-v]+|p-v|^2w}{\big| \al\bR\big[\frac{p-v}{|p-v|} +|p-v|w\big] \big|^2}, \qquad w:= \frac{\al}{2}\bR[v] -\frac{v}{2},
	\end{equation*}
	which is checked to be of class $C^1(U)$ and conformal with the definition $F(v)=v$ (one has $dF_v=\frac{1}{\al}\bR$). By the regularity Theorem 3.1 in \cite{LiSa} we conclude that $F$ is actually smooth.\\
	
	\noindent The inverse map $I_{1,v}^{-1}$ has differential
	\begin{equation*}
	d(I^{-1}_{1,v})_q=\frac{1}{|q|^2}\bigg( id -\frac{2}{|q|^2}q\tens q \bigg).
	\end{equation*}
	Hence taking $q=(1/2,t,0)\in r_{-v/2}$, $e_3=(0,0,1)$ and $X\in (T_q(r_{-v/2}))^\perp$ we have $d(I^{-1}_{1,v})_q(e_3)=\frac{1}{|q|^2}e_3$ and thus
	\begin{equation} \label{ref16'}
	\bigg( \frac{d(I^{-1}_{1,v})_q(X)}{|d(I^{-1}_{1,v})_q(X)|} \bigg)_3:= \frac{\lgl d(I^{-1}_{1,v})_q(X), e_3\rgl}{|d(I^{-1}_{1,v})_q(X)|}=\frac{\lgl d(I^{-1}_{1,v})_q(X),d(I^{-1}_{1,v})_q(e_3)\rgl}{|d(I^{-1}_{1,v})_q(X)||d(I^{-1}_{1,v})_q(e_3)|}=\frac{\lgl X, e_3 \rgl}{|X|}=: \bigg(\frac{X}{|X|} \bigg)_3,
	\end{equation}
	that is the "vertical" component of a vector is preserved.
	Moreover for $q=(1/2,t,0)\in r_{-v/2}$ we have
	\begin{equation} \label{ref17'}
	\frac{d(I^{-1}_{1,v})_q((-1,0,0))}{|d(I^{-1}_{1,v})_q((-1,0,0))|} =(-1,0,0)+\frac{q}{|q|^2}=I^{-1}_{1,v}(q),
	\end{equation}
	that is the tangent vector $(-1,0,0)$ at a point $q=I_{1,v}(p)$ is precisely mapped into the tangent vector equal to $p$ at $I_{1,v}^{-1}(q)=p$.\\ 
	Since spherical inversions preserve the orientation, the field $\frac{d(I^{-1}_{1,v})_q(co_{\Pi_1})}{|d(I^{-1}_{1,v})_q(co_{\Pi_1})|}$ coincides with $co_{F\circ \Phi}$. Hence putting together \eqref{ref-1}, \eqref{ref0}, \eqref{ref16'} and \eqref{ref17'} and choosing $\al,\be,\ga$ sufficiently small, we have the thesis.
\end{proof}

\begin{cor} \label{cop'}
	If $\Phi\in\cF$ is such that $G(\Phi)<2\pi$, then there is $\Phi'\in\cF$ such that $\cW(\Phi')<\cW(\Phi)$.
	In particular a minimizer $\Phi$ of problem \eqref{problemS1} must satisfy $co_\Phi(p)\equiv p$.
\end{cor}

\begin{proof}
	Apply Lemma \ref{cop} recalling that the quantity $\cW+G$ is conformally invariant, so that if $G$ increases then $\cW$ must decrease.
\end{proof}

\begin{proof}[Proof of Theorem \ref{nonexS1}]
	Assume by contradiction that a minimizer $\Si$ exists. Then by Corollary \ref{cop'} the conormal $co$ of $\Si$ is identically equal to the field $p$. Let $\Si^{ext}:=\Si\cup \{z=0,x^2+y^2\ge1  \}\in C^{1,1}$
	%
	.
	Now two possibilities can occur.\\
	Suppose first there exists $\bar{p}\in int(D)$ such that $\bar{p}\not\in\Si$, where $D$ is the closed disk enclosed by $\S^1$. Then $I_{1,\bar{p}}(\Si^{ext})\cup \{0\}=:\Si'$ is a well defined surface of class $C^{1,1}$ without boundary with genus $\cg$. Also $\Si'\supset D$, then $\Si'$ cannot be a minimizer for the Willmore energy among closed surfaces of genus $\cg$, otherwise $\Si'$ would be analytic (\cite[Theorem I.3]{Ri08}) and equal to the plane containing $D$. Hence $\cW(\Si')>\be_\cg$, and then $\cW(\Si)=\cW(\Si^{ext})>\be_\cg -4\pi $. Since the infimum of our problem is $\le \be_\cg-4\pi$, this implies that $\Si$ could not be a minimizer.\\
	Suppose now the other case: $D\con\Si$. In this case the whole plane containing $D$ is contained in $\Si^{ext}$, which has genus $\cg\ge1$, then there exists a point $q\in\Si^{ext}$ with multiplicity $\ge2$. Now let $x\in\R^3\sm\Si^{ext}$, then $\Si':=I_{1,x}(\Si^{ext})\cup\{0\}$ is a $C^{1,1}$ closed surface of genus $\cg$ with a point of multiplicity $\ge2$, then $\cW(\Si')\ge8\pi$ and $\cW(\Si)=\cW(\Si^{ext})\ge 4\pi$. Since the infimum of our problem is $\le \be_\cg-4\pi<4\pi$, this implies that $\Si$ could not be a minimizer.\\
	Finally, for any $\ep>0$ we know that
	$$\inf_{\Si\in\cF} \cW =\inf_{\Si\in\cF,G(\Si)\ge2\pi-\ep}\cW,$$
	then, since $\inf_{\Si\in\cF} \cW\le \be_\cg-4\pi$, by the above argument we conclude that
	$$\inf_{\Si\in\cF} \cW = \inf_{\Si\in\cF,G(\Si)=2\pi}\cW=\be_\cg-4\pi.$$
\end{proof}

\textcolor{white}{text}


\section{Lower bound estimates and consequences}

\noindent We derive lower bounds on immersed surfaces with boundary with a point of multiplicity greater than one. From a variational viewpoint, this may allow us to restrict the set of competitors of a minimization problem to embedded surfaces.\\

\noindent Let $0<\si<\ro$ and $p_0\in\R^3$. Let us recall here a monotonicity formula for varifolds with boundary, which can be obtained by integrating the tangential divergence of the field $X(p)=\big( \frac{1}{|p-p_0|^2_\si} - \frac{1}{\ro^2} \big)_+(p-p_0)$, where $|p-p_0|^2_\si=\max\{\si^2,|p-p_0|^2\}$ and $(\cdot)_+$ denotes the positive part (see \cite{NoPo20} and \cite{RiLI}). If $V$ is an integer rectifiable curvature varifold with boundary with bounded Willmore energy, with $\nu$ the induced measure in $\R^3$, and generalized boundary $\si_V$, it holds
\begin{equation} \label{monot}
A(\si)+\int_{B_\ro(p_0)\sm B_\si (p_0) } \bigg| \frac{\vec{H}}{2} +\frac{(p-p_0)^\perp}{|p-p_0|^2} \bigg|^2\,d\nu(p) = A(\ro),
\end{equation}
where
\begin{equation*}
A(\ro):= \frac{\nu(B_\ro(p_0))}{\ro^2}+\frac{1}{4}\int_{B_\ro(p_0)} |H|^2\,d\nu(p)+ R_{p_0,\ro},
\end{equation*}
and
\begin{equation*}
\begin{split}
	R_{p_0,\ro}&:= \int_{B_\ro(p_0)} \frac{\lgl \vec{H}, p-p_0\rgl}{\ro^2}\,d\nu(p) + \frac{1}{2}\int_{B_\ro(p_0)} \bigg( \frac{1}{|p-p_0|^2}-\frac{1}{\ro^2} \bigg)(p-p_0) \,d\si_V(p)=\\& =: \int_{B_\ro(p_0)} \frac{\lgl \vec{H}, p-p_0\rgl}{\ro^2}\,d\nu(p) + T_{p_0,\ro}.
\end{split}
\end{equation*}
In particular the function $\ro\mapsto A(\ro)$ is monotonically nondecreasing.\\

\noindent Let us denote by $I:\R^3\sm\{0\}\to\R^3\sm\{0\}$ the standard spherical inversion $I(x)=\frac{x}{|x|^2}$.\\

\noindent Also, we need to introduce the minimization problem of the Willmore energy among complete unbounded surfaces with a straight line as boundary, leading to some tools that we will use in the following.\\
Let $r\con\R^3$ be a straight line and fix an integer $\cg\ge 1$. We consider the minimization problem
\begin{equation} \label{problemr}
\min \{\cW(\Si)| \Si\in\cA \},
\end{equation}
where we say that $\Si\in \cA$ if
\begin{itemize}
	\item $\Si$ is a smooth immersed asymptotically flat surface with one end of genus $\cg$ with boundary $r$.
\end{itemize}
Recalling the definitions given in the Introduction, one has $D(\Si)<+\infty$ for $\Si\in \cA$.
We are going to see that problem \eqref{problemr} is strongly related to another minimization problem on the following family of surfaces.\\
We say that $\Si\in \cB$ if
\begin{enumerate}
	\item $\Si=\Phi(\Si_\cg)$ and $\Phi|_{\pa\Si_\cg}:\pa\Si_\cg\to \S^1$,
	\item there exists $p\in \S^1$ such that $\Phi^{-1}(p)=\{p_0\}$,
	\item $\Phi$ is smooth on $\Si_\cg\sm\{p_0\}$ and $\Phi$ is globally of class $C^{1,1}$.
\end{enumerate}
where $\S^1$ is a given round circle. Observe that if $\Si\in\cB$ and $p\in \pa \Si$ verifies $m(p)=1$, then $I(\Si\sm\{p\})\in\cA$ up to translation. Analogously if $\Si\in\cA$, then $I(\Si)\cup\{0\}\in\cB$.\\

\noindent Let us remind that by the classical Gauss-Bonnet Theorem, if $S$ is a smooth compact oriented surface with piecewise smooth boundary $\pa S$ positively oriented with respect to $S$, then the quantity $F(S):=\int_SK_S+G(S)+\al(S)$ is a topological invariant. Here $K_S$ is the Gaussian curvature and $\al(S)$ is the sum of the angles described by the tangent vector of a positive parametrization of $\pa S$ at the corners of $\pa S$, such angles being counted with the given orientation of $S$ and $\pa S$.

\begin{lemma} \label{sphericalr}
	If $\Si\in \cB$ and $I$ is a spherical inversion with center at $p\in \S^1$ point of multiplicity $1$ in $\Si$, then $I(\Si\sm\{p\})\in\cA$ for some line $r$ and
	\begin{equation*}
	(\cW+G)(\Si)=\cW(I(\Si\sm\{p\}))+2\pi.
	\end{equation*}
	Similarly if $\Si\in \cA$ and $I$ is a spherical inversion with center at a point $p\not\in \Si$, then up to isometry $I(\Si)\cup\{0\}\in\cB$ and
	\begin{equation*}
	(\cW+G)(I(\Si)\cup\{0\})=\cW(\Si)+2\pi.
	\end{equation*}
\end{lemma}

\begin{proof}
	It is enough to prove the first part of the statement, being the second part completely analogous. So let $\Si\in\cB$. A rotation and a translation yields a surface still denoted by $\Si$ with boundary $\{ (x-1)^2+y^2=1,z=0 \}$ and with the origin point of multiplicity $1$ in $\Si$. The standard inversion $I$ maps $\pa\Si$ onto the line $\{(-1/2,t,0)|t\in\R \}$ and clearly $I(\Si\sm\{p\})\in\cA$.\\
	Consider $\Si_r:=\overline{\Si\sm B_r(0)}$ for $r$ sufficiently small so that $\overline{\Si\cap B_r(0)}$ is homeomorphic to a closed disk. Let $r$ be fixed for the moment. The boundary of $\Si_r$ is just piecewise smooth, but we can approximate the surface $\Si_r$ (and then $I(\Si_r)$) by surfaces with smooth boundary diffeomorphic to $\Si_r$, so that, since as shown in \cite{Ch} the quantity $(|H|^2-K)g$ is pointwise conformally invariant, we get that
	\begin{equation} \label{ref20}
	\int_{\Si_r} |H_{\Si_r}|^2-K_{\Si_r} = \int_{ I(\Si_r)} |H_{I(\Si_r)}|^2-K_{I(\Si_r)},
	\end{equation}
	for any $r$ small. Observe that $ \int_{\Si_r} |H_{\Si_r}|^2\to \int_{\Si} |H_{\Si}|^2 $ and $\int_{ I(\Si_r)} |H_{I(\Si_r)}|^2\to \int_{ I(\Si\sm\{p\})} |H_{I(\Si\sm\{p\})}|^2$ as $r\to0$. So we now look at the change of the integral in the Gaussian curvature studying the Gauss-Bonnet topological invariant $F(S):=\int_S K_S+ G(S)+ \al(S)$, where $S$ is a smooth surface with piecewise smooth boundary and $\al(S)$ in the term taking into account the oriented angles determined by the possible corners of $\pa S$.\\
	Up to the choice of the orientation of $\Si$, we can assume $\S^1$ to be positively oriented with respect to $\Si$ with the usual counterclockwise orientation. As $r\to0$ the boundary curve $\pa B_r(0)\cap\Si$ is close in $C^2$ norm to a half circumference lying in $T_0\Si$, by construction oriented with its curvature vector $\vec{k}_r$ such that $r\lgl\vec{k}_r,co_{\Si_r}\rgl \to 1$ uniformly as $r\to0$. Then $G(\Si_r)\to -\pi+G(\Si)$ as $r\to0$. Also by the choice of the orientation we get $\al(\Si_r)\to\pi$ as $r\to0$.\\
	The boundary of $\Si_r$ is mapped by $I$ onto a segment $s_r:=\{(-1/2,t,0):|t|\le t_r \}$ union with $I(\pa B_r(0)\cap\Si)$ which is close in $C^2$ norm to a half circumference as $r\to0$. Since $I$ preserves the orientation and maps point closer to the origin to point farther from the origin (and viceversa) we now have that the curvature vector $\vec{k}'_r$ of $I(\pa B_r(0)\cap\Si)$ is such that $\bigg\lgl\frac{\vec{k}'_r}{|\vec{k}'_r|},co_{I(\Si_r)}\bigg\rgl \to -1$ uniformly on $I(\pa B_r(0)\cap\Si)$ as $r\to0$. Then $G(I(\Si_r))\to \pi$ as $r\to0$. Also by the orientation preserving property we still have $\al(I(\Si_r))\to\pi$ as $r\to0$.\\
	Therefore adding $F(\Si_r)=F(I(\Si_r))$ to equation \eqref{ref20} and passing to the limit $r\to0$ we get the claim.
\end{proof}

\begin{lemma} \label{norma}
	For any $\om<4\pi$ exists $\ep>0$ such that if $g:\Si_\cg\to \R^3$ is an immersion with $g(\pa\Si_\cg)=\S^1$, outward conormal field $co$ and
	\begin{equation*}
	\exists p_0\in\R^3:\quad \sharp g^{-1}(p_0)\ge 2,
	\end{equation*}
	\begin{equation*}
	\| co(p) - p \|_{L^2(\S^1)}\le\ep,
	\end{equation*}
	then $\cW(g)\ge \omega$.
\end{lemma}

\begin{proof}
	By approximation taking a small perturbation of the surface we can prove the statement for $p_0\not\in \S^1$. Let us assume by contradiction that
	\begin{equation} \label{contr}
	\begin{split}
		\exists \Si_n:=g_n(\Si_\cg): \quad &\|co_n(p)-p\|_{L^2(\S^1)}\le \frac{1}{n}, \\
		& \exists p_n\in \Si_n\sm \S^1:\quad \sharp g_n^{-1}(p_n)\ge 2,\\
		& \cW(\Si_n)\le \om <4\pi.
	\end{split}
	\end{equation}
	Up to a small modification of the sequence which preserves \eqref{contr}, we can assume that for any $n$ there is $q\in \S^1$ such that $\sharp g_n^{-1}(q)=1$. Then we consider the sequence $\Si'_n:= I_{1,q}(\Si_n\sm\{q\})$, that, up to isometry, is an asymptotically flat manifold with one end such that
	\begin{equation*}
	\begin{split}
		&\Ga'_n:=\pa\Si'_n = \{ (X_n,y,0)|y\in\R \} \qquad X_n\in\R_{>0},\\
		& \te'_n(0)\ge 2,\\
		& 0\not\in \Ga'_n,\\
		&\cW(\Si'_n) = \cW(\Si_n) + G(\Si_n) -2\pi,
	\end{split}
	\end{equation*} 
	where $\te'_n(p)=\lim_{r\searrow0} \frac{|\Si'_n\cap B_r(0)|}{\pi r^2}$ is the multiplicity function and the last equality is given by Lemma \ref{sphericalr}.\\
	Consider now a blow up sequence $\Si''_n:=\frac{\Si_n}{r_n}$ for $r_n\searrow0$ such that
	\begin{equation*}
	\begin{split}
		&i)\, \int_{\Si'_n\cap B_{r_n}(0)} |\sff'_n|^2 \le \frac{1}{n},\\
		&ii)\, \Si'_n\cap B_{r_n}(0) =\bigcup_{i=1}^{\te_n'(0)} D_{n,i},\quad D_{n,i}\simeq \mbox{ Disk},\\
		&iii)\, \forall n\,\forall i=1,...,\te'_n(0) \quad\exists u_{n_i}:\Om_{n,i}\con L_{n,i} \to \R \,\, \mbox{s.t. }\\
		&\qquad L_{n,i} \mbox{ plane},\\
		&\qquad graph(u_{n,i})=D_{n,i},\\
		&\qquad |u_{n,i}|,|\nabla u_{n,i} | \le \frac{1}{n},\\
		& iv)\, \Ga''_n:=\pa \Si''_n =\{ (R_n,y,0)| y\in\R \}\qquad R_n:=\frac{X_n}{r_n}\to\infty.
	\end{split}
	\end{equation*}
	Then up to subsequence $\Si''_n$ converges in the sense of varifolds to the integer rectifiable varifold $\mu=\bv\bigg(\bigcup_{i=1}^M \Pi_i, \te \bigg)$, where each $\Pi_i$ is a plane passing through the origin and $M\ge2$ or $M=1,\te\ge2$.\\
	Now we exploit the monotonicity formula \eqref{monot} with $p_0=0$. Calculating $T_{0,\ro}$ on the sequence $\Si''_n$ gives
	\begin{equation*}
	\forall n\,\exists\lim_{\ro\to\infty} T_{0,\ro}(\Si''_n) = \frac{1}{2}\int_{\Ga''_n}  \frac{\lgl p, co''_n(p)\rgl}{|p|^2}  \,d\cH^1(p),
	\end{equation*}
	indeed
	\begin{equation*}
	\begin{split}
		\bigg|\frac{1}{2\ro^2}\int_{\Ga''_n\cap B_\ro(0)} \lgl p, co''_n(p)\rgl  \,d\cH^1(p)\bigg| &\le \bigg|\frac{1}{2\ro^2}\int_{\Ga''_n\cap B_\ro(0)} \lgl (R_n,y,0), ((co''_n(p))_1,0,(co''_n(p))_3)\rgl  \,d\cH^1(p)\bigg| \le\\&\le \frac{R_n}{2\ro^2}\cH^1(\Ga''_n\cap B_\ro(0))\le \frac{R_n}{\ro}\xrightarrow[\ro\to\infty]{}0.
	\end{split}
	\end{equation*}
	Also on $\Si''_n$ we have for any $0<\si<\ro$ that
	\begin{equation*}
	\begin{split}
		\int_{\Si''_n\cap B_\ro(0)} \frac{\lgl \vec{H}''_n,p\rgl}{\ro^2} &= 	\int_{\Si''_n\cap B_\si(0)} \frac{\lgl \vec{H}''_n,p\rgl}{\ro^2} + 	\int_{\Si''_n\cap B_\ro(0)\sm B_\si(0)} \frac{\lgl \vec{H}''_n,p\rgl}{\ro^2} \le\\&\le  \frac{1}{\ro} \int_{\Si''_n\cap B_\si(0)} |H_n''| + \frac{|\Si''_n\cap B_\ro(0)|^{1/2}}{\ro} \cW(\Si''_n\sm B_\si(0)),
	\end{split}
	\end{equation*}
	hence letting first $\ro\to\infty$ and then $\si\to\infty$, since $\lim_{\ro\to\infty}\frac{|\Si''_n\cap B_\ro(0)|^{1/2}}{\ro} = (\pi/2)^{1/2}$, we have
	\begin{equation*}
	\lim_{\ro\to\infty}\int_{\Si''_n\cap B_\ro(0)} \frac{\lgl \vec{H}''_n,p\rgl}{\ro^2}=0.
	\end{equation*}
	So by monotonicity of $A_{\Si''_n}$, that is $A$ of \eqref{monot} calculated on $\Si''_n$, we have
	\begin{equation*}
	\begin{split}
		\exists \lim_{\ro\to\infty} A_{\Si''_n}(\ro) &= \lim_{\ro\to\infty} \frac{|\Si''_n\cap B_\ro(0)|}{\ro^2} + \frac{\cW(\Si''_n)}{4} + \frac{1}{2}\int_{\Ga''_n}  \frac{\lgl p, co''_n(p)\rgl}{|p|^2}  \,d\cH^1(p) =\\&= \frac{\pi}{2} +  \frac{\cW(\Si''_n)}{4} + \frac{1}{2}\int_{\Ga''_n}  \frac{\lgl p, co''_n(p)\rgl}{|p|^2}  \,d\cH^1(p) =\\
		&=  \frac{\pi}{2} +  \frac{\cW(\Si''_n)}{4} + \frac{1}{2}\int_{\R} \frac{R_n (co''_n(p))_1}{R^2_n +y^2}\,dy\le \\
		&\le  \frac{\pi}{2} +  \frac{\cW(\Si''_n)}{4} + \frac{1}{2}\int_{\R} \frac{1}{1 +u^2}\,du =\pi + \frac{\cW(\Si''_n)}{4}.
	\end{split}
	\end{equation*}
	By the convergence $\Si''_n\to \mu$ in the sense of varifolds, the quantity $A(\ro)$ is lower semicontinuous for almost all $\ro>0$, i.e. $\liminf_n A_{\Si''_n}(\ro)\ge A_\mu(\ro)$. In fact the first and second summands in the definition of $A(\ro)$, that is the mass and the Willmore energy in the ball $B_\ro(0)$, are lower semicontinuous, while by continuity of the first variation under varifold convergence, the summand $R_{0,\ro}$ is continuous in the limit $n\to\infty$. Hence if $A_\mu$ is the $A$ of \eqref{monot} calculated on the varifold $\mu$, using also monotonicity we have
	\begin{equation} \label{ref42}
	\begin{split}
		\mbox{for a.a. }\ro>0:\qquad A_\mu(\ro)&\le \liminf_{n\to\infty} A_{\Si''_n}(\ro) \le   \liminf_{n\to\infty} \lim_{\ro\to\infty}  A_{\Si''_n}(\ro) \le \liminf_{n\to\infty} \pi +\frac{\cW(\Si''_n)}{4} = \\
		&= \pi + \liminf_{n\to\infty} \frac{\cW(\Si'_n)}{4} =\\
		&= \pi + \liminf_{n\to\infty} \frac{\cW(\Si_n)}{4} < 2\pi,
	\end{split}
	\end{equation}
	where we used that by the absurd hypothesis $G(\Si_n)\to 2\pi$.\\
	Passing to the limit in \eqref{ref42} we find
	\begin{equation*}
	2\pi \le \lim_{\ro\to\infty} \frac{\mu(B_\ro(0))}{\ro^2} = \lim_{\ro\to\infty}  A_\mu(\ro) < 2\pi,
	\end{equation*}
	which gives the desired contradiction.
\end{proof}

\begin{remark} \label{remcirc}
	Let us mention that if $\Ga$ is not a circumference, arguing as in Lemma \ref{norma}, one can perform similar estimates and calculate the integral $\int_{\Ga''_n} \frac{\lgl p, co''_n\rgl}{|p|^2}$ by the change of variable $p=I_{1,0}(q)$. One ends up with the following statement, whose proof is omitted here since we will not need this in the following (see \cite{PozzPhD}).\\
	Let $\Ga$ be an embedded planar closed smooth curve and let $\nu_\Ga$ be the unit outward conormal of the planar region enclosed by $\Ga$. Then for any $\om<4\pi$ there exists $\ep>0$ such that if $\Phi:\Si_\cg\to \R^3$ is an immersion with $\Phi(\pa\Si_\cg)=\Ga$, outward conormal field $co$ and
	\begin{equation*}
	\exists p_0\in\R^3:\quad \sharp \Phi^{-1}(p_0)\ge 2,
	\end{equation*}
	\begin{equation*}
	\| co - \nu_\Ga \|_{L^2(\Ga)}\le\ep,
	\end{equation*}
	then $\cW(\Phi)\ge \om$.
\end{remark}

\begin{cor} \label{lowbound}
	If $\Phi:\Si_\cg\to \R^3$ is an immersion with $\Phi(\pa\Si_\cg)=\S^1$ and there is $ p_0\in\R^3:\sharp \Phi^{-1}(p_0)\ge 2$, then $\cW(\Phi)\ge 4\pi$.
\end{cor}

\begin{proof}
	By approximation taking a small perturbation of the surface we can prove the statement for $p_0\not\in \S^1$. Call $\Si=\Phi(\Si_\cg)$. For any $\ep\in(0,1)$ by Lemma \ref{cop} we get the existence of a surface $\Phi'(\Si_\cg)=\Si'$ with boundary $\S^1$ such that
	\begin{equation*}
	\begin{split}
		&i)\, \Phi'=F\circ \Phi \qquad F\mbox{ conformal,}\\
		&ii)\, G(\Si')>G(\Si),\\
		&iii)\, \exists p'\in\Si'\sm \S^1:\qquad \sharp \Phi'^{-1}(p')\ge2,\\
		&iv)\, \| (co_{\Si'})(p)-p\|_{L^2(\S^1)}\le \ep.\\
	\end{split}
	\end{equation*}
	Then by $i),ii)$ and by Lemma \ref{norma} we have
	\begin{equation*}
	\cW(\Si)\ge\cW(\Si')\ge \om \qquad \forall\om<4\pi.
	\end{equation*}
\end{proof}

\begin{cor} \label{corequiv}
	The minimization problem $\min\{ \cW(\Si)|\Si\in\cA\}$ is equivalent to both the following minimization problems
	\begin{equation} \label{ref21}
	\min \{ (\cW+G)(\Si)| \Si\in \cB \},
	\end{equation}
	\begin{equation} \label{probr}
	\min \{\cW(\Si)| \Si\in\cA, \Si \mbox{ embedded.} \}.
	\end{equation}
	In fact, it holds that if $\Si\in\cA$ has a point with multiplicity $\ge2$, then $\cW(\Si)\ge4\pi$.
\end{cor}

\begin{proof}
	Equivalence with problem \eqref{ref21} follows from Lemma \ref{sphericalr}. Now suppose $\Si\in\cA$ has a point with multiplicity $\ge2$. By Lemma \ref{sphericalr} the surface $\Si':=I(\Si)\cup\{0\} \in \cB$ verifies
	\begin{equation*}
	\cW(\Si)=\cW(\Si')+ G(\Si')-2\pi.
	\end{equation*}
	If $G(\Si')=2\pi$, then by Corollary \ref{lowbound} follows that $\cW(\Si)\ge4\pi>\be_\cg-4\pi$. If $G(\Si')<2\pi$, for any $\ep>0$ we can apply Lemma \ref{cop} in order to get a new surface $F(\Si')$ such that
	\begin{equation*}
	\cW(\Si)=\cW(F(\Si'))+ G(F(\Si'))-2\pi>\cW(F(\Si'))-\ep\ge4\pi-\ep,
	\end{equation*}
	where the second inequality follows by Corollary \ref{lowbound}. Letting $\ep\to0$ the proof is completed.
\end{proof}

\begin{cor} \label{lb}
	There exists $\ep_0>0$ such that if $\Ga$ is a planar curve such that there is a diffeomorphism $F:\R^3\to\R^3$ with $F(\Ga)=\S^1$ and
	$$\|F-id|_{\R^3}\|_{C^2(\R^3)}\le \ep_0,$$
	then
	\begin{equation*}
	\inf\{ \cW(\Phi): \Phi(\pa\Si_\cg)=\Ga \}=\inf\{ \cW(\Phi): \Phi(\pa\Si_\cg)=\Ga,\, \Phi \mbox{ embedding} \}.
	\end{equation*}
	Indeed if $\Phi:\Si_\cg\to\R^3$ is an immersion with $\Phi(\pa\Si_\cg)=\Ga$ and $$\exists p_0\in\R^3:\sharp \Phi^{-1}(p_0)\ge 2,$$
	then $\cW(\Phi)\ge 4\pi - \eta(\ep_0)$ for some $\eta(\ep_0)\to0 $ as $\ep_0\to0$.
\end{cor}

\begin{proof}
	If $\Phi(\pa\Si_\cg)=\Ga$, for $\ep_0$ sufficiently small there is $\eta(\ep_0)$ such that $\eta(\ep_0)\to0 $ as $\ep_0\to0$ for which
	\begin{equation*}
	\cW(\Phi)+\eta(\ep_0)\ge \cW(F\circ \Phi).
	\end{equation*}
	Then the thesis follows by Corollary \ref{lowbound} and the fact that $\inf\{ \cW(\Phi): \Phi(\pa\Si_\cg)=\Ga \}\le \be_\cg-4\pi<4\pi$.
\end{proof}

\noindent Putting together Theorem \ref{nonexS1} with Corollary \ref{lowbound} we get another non-existence result:

\begin{cor} \label{cors1equiv}
	For any genus $\cg\ge1$, the minimization problem
	\begin{equation*}
	\min\{ \cW(\Phi)|\Phi:\Si_\cg\to \R^3 \mbox{ smooth embedding, } \Phi(\pa\Si_\cg)=\S^1 \}
	\end{equation*}
	has no minimizers and the infimum equals $\be_\cg-4\pi$.
\end{cor}

\textcolor{white}{text}
\section{Varifold approach to existence results}

\noindent Now we want to study the generalized problem on classes $\cC(\Si_\cg)$. Let us explain the organization of the section. In Subsection \ref{sub1} we apply Simon's ambient approach to derive information on limits of minimizing sequences. In Subsection \ref{sub2} we come back to the study of the classical problem in the case of asymptotically flat surfaces having a straight line as boundary. Then in the third subsection we conclude the study of the generalized problem on the class $\cC(\Si_1)$; to this aim we will use the results of Subsection \ref{sub2}.\\
Note that in the following we will also use definitions and facts recalled in Appendix A and Appendix B.\\

\textcolor{white}{text}

\subsection{Generalized problem and Simon's direct method}\label{sub1}

\noindent Let $ \Ga $ be a closed planar smooth embedded curve in $\R^2=\{z=0\}\con\R^3$ and let $D$ be the bounded planar region enclosed by $\Ga$. Fix an integer $\cg\ge 1$. Recall that we say that a continuous proper map $\Phi:S\to\R^3$ is a \emph{branched immersion} if $\Phi|_{S\sm\{y_1,...,y_P\}}$ is a smooth immersion for some $y_1,...,y_P\in S$. Denote by $\Imm(\Phi)$ the image varifold induced by a possibly branched immersion $\Phi$.\\
In the Introduction we have defined the classes $\cC(\Si_\cg)$ and $\cC_{imm}(\Si_\cg)$.\\

\noindent We now study the minimization problem
\begin{equation} \label{problemmin}
\min \big\{\cW(V)| V\in \cC(\Si_\cg) \big\}.
\end{equation}

\begin{remark} \label{remset}
	The use of the set $\cC(\Si_\cg)$ is essentially technical and fundamental for the proof of the following Proposition \ref{simon}. It would be reasonable to define problem \eqref{problemmin} on every immersed surface and then restrict it to smaller classes by means of some result like Corollary \ref{lowbound}. However, analogous results are absent for arbitrary curves, although much likely to be true at least for convex ones. This is the reason why we say that the choice of the class $\cC$ is natural from a variational perspective.
\end{remark}

\noindent From now on, assuming that $\inf_{\cC(\Si_\cg)}\cW= \inf_{\cC_{imm}(\Si_\cg)} \cW$, we will denote by $\Phi_n$ a sequence of embeddings which is minimizing for the problem \eqref{problemmin} and by $\Si_n$ the sequence of integer rectifiable varifolds in $\R^3$ induced by each $\Phi_n$.\\
\noindent A point $\xi\in\R^3$ is called a \emph{bad point for the sequence $\Si_n$ with respect to a parameter $\ep>0$} if
\begin{equation*}
\lim_{r\searrow0}\liminf_{n\to+\infty} D(\Si_n\cap B_r(\xi))\ge \ep^2.
\end{equation*}
\noindent If a point is not bad, we say that it is \emph{good}. In the rest of the subsection we will deeply make use of the techniques developed in \cite{SiEX} and in \cite{ScBP}, so in the following proofs we mainly remark the differences and the adaptations of such techniques to our case.

\begin{prop}\label{simon}
	Assume that $\inf_{\cC(\Si_\cg)}\cW= \inf_{\cC_{imm}(\Si_\cg)} \cW$. Up to subsequence, $\Si_n$ converges in the sense of varifolds to an integer rectifiable varifold $\Si$ with boundary.
	The set of bas points $\{\xi_j \}$ of $\Si_n$ with respect to $\ep$ is finite, and it holds that
	\begin{equation} \label{diffeom}
	\exists\si_n\searrow 0 :\quad\Si\sm \cup_j B_{\si_n}(\xi_j) = \overline{\Phi}_n(S), \quad S\simeq \Si_n \sm \cup_j B_{\si_n}(\xi_j),
	\end{equation}
	for smooth immersions $\overline{\Phi}_n:S\to\R^3$ of a two dimensional manifold $S$ with boundary.\\
	The varifold $\Si$ is induced by a possibly branched immersion $\Phi:\Si_{\cg'}\to\R^3$. Moreover, if $\inf_{\cC(\Si_{\cg'})}\cW= \inf_{\cC_{imm}(\Si_{\cg'})} \cW$, $\Si$ is a minimizer of the problem \eqref{problemmin} in its own class of genus $\cg'$ and $\cg'\le\cg$.
\end{prop}

\begin{remark}
	We remark that the compactness result of Proposition \ref{simon}, i.e. the extraction of a subsequence converging in the sense of varifolds, is new in the setting of surfaces with boundary $\R^3$ having bounded Willmore energy and no prescription on the conormal. In fact, this property follows from the minimizing assumption on the sequence and not only from the bound in energy.
\end{remark}

\begin{proof}[Proof of Proposition \ref{simon}]
	By compactness properties of varifolds (Appendix A), in order to establish the varifold convergence we just have to prove a uniform bound on the sequence of the areas $|\Si_n|$. Indeed by Gauss-Bonnet
	$$\cW(\Si_n)=\frac{1}{4}D(\Si_n)+\frac{1}{2}\bigg( 2\pi\chi(\Si_n)-\int_\Ga (k_g)_n \bigg)$$
	we get a uniform bound on $D(\Si_n)$; while for any $n$ the varifold boundary of $\Si_n$ has total variation bounded by the measure $\cH^1\res\Ga$. Also, the uniform bound on $D(\Si_n)$ implies that there are at most finitely many bad points with respect to any fixed parameter $\ep>0$. So for an estimate on the areas, let us first show a bound on the diameters of each $\Si_n$. Since $\Si_n$ is embedded, passing to the limit $\si\to0$ and $\ro\to\infty$ in the monotonicity formula with boundary \eqref{monot} follows that
	\begin{equation*}
	\pi+\int_{\Si_n} \bigg| \frac{H_n}{2}+\frac{(p-p_0)^\perp}{|p-p_0|^2} \bigg|^2=\frac{1}{2}\int_\Ga \bigg\lgl \frac{p-p_0}{|p-p_0|^2},co_n \bigg\rgl+\frac{\cW(\Si_n)}{4},
	\end{equation*}
	for any $p_0\in\Si_n\sm \Ga$. Since $ \int_\Ga \lgl \frac{p-p_0}{|p-p_0|^2},co_n \rgl \le \frac{\cH^1(\Ga)}{dist(p_0,\Ga)} $ and $\cW(\Si_n)\le \be_\cg-4\pi+\ep_n$ for some $\ep_n\to0$, we get
	$$ \frac{8\pi-\be_\cg -\ep_n}{2}\le \frac{\cH^1(\Ga)}{dist(p_0,\Ga)}. $$
	Now for $n$ sufficiently big the quantity $8\pi-\be_\cg-\ep_n$ is strictly positive, then
	$$ dist(p_0,\Ga)\le \frac{2\cH^1(\Ga)}{8\pi-\be_\cg-\ep_n}\le C(\cg,\Ga). $$
	This is true for any $p_0\in \Si_n\sm \Ga$, and since $\Ga$ is fixed we get a bound on the diameters $diam(\Si_n)\le C$ which is uniform in $n$.\\
	So let us say that for any $n$ we have $\Si_n\con\con B_R(0)$. We can estimate the area of $\Si_n$ in the same fashion used in the proof of Proposition \ref{ref27}; so consider $\Si'_n:=\frac{1}{R}\Si_n \con\con B_1(0)$ and we have
	\begin{equation*}
	\begin{split}
		2|\Si'_n| &= \int_{\Si'_n} div_{\Si'_n} x =-2\int_{\Si'_n} \lgl \vec{H}'_n,x\rgl +\int_{\pa\Si'_n} \lgl x, co'_n(x)\rgl \,d\cH^1=\\&= |\Si'_n| -\int_{\Si'_n} (1-|x^\perp|^2)-\int_{\Si'_n} |\vec{H}'_n+x^\perp|^2+\cW(\Si'_n)+\int_{\pa \Si'_n} \lgl x, co'_n(x)\rgl \,d\cH^1 \le \\
		&\le |\Si'_n|+\cW(\Si_n)+\cH^1(\pa \Si'_n)\le\\
		&\le |\Si'_n|+\cW(\Si_n)+\frac{\cH^1(\Ga)}{R}.
	\end{split}
	\end{equation*}
	Hence $|\Si_n|=R^2|\Si'_n|$ is uniformly bounded in $n$, and therefore we get varifold convergence up to subsequence to some limit $\Si$ with boundary $\Ga$.\\
	Now we are going to exploit the regularity theory developed in \cite{SiEX}, which is based on Lemma 2.1 of \cite{SiEX}. So let $\ep>0$ and denote by $\xi_1,...,\xi_P$ the related bad points. We can apply the arguments in \cite{SiEX} as did in \cite{ScBP} without modifications to get regularity in any good point $\xi\in \supp(\Si)\sm\Ga$, thus getting $\supp(\Si)\in C^\infty$ at such points, in the sense that $\supp(\Si)$ is locally union of graphs of $C^\infty$ functions near those points. Here we are using the hypothesis that $\inf_{\cC(\Si_\cg)}\cW= \inf_{\cC_{imm}(\Si_\cg)} \cW$, as the biharmonic comparison arguments of \cite{SiEX} might not preserve the embeddedness of $\Si_n$. The same method can applied also in the case of good points $\xi\in\Ga$, except that now the sets $d_{i,k}$ in Lemma 2.1 of \cite{SiEX} may touch $\Ga$ (see Appendix B). However the biharmonic comparison method of Theorem 3.1 of \cite{SiEX} can be applied in the analogous way as stated in Appendix B by comparison with biharmonic functions with graphs passing through $\Ga$; hence the usual estimates are achievable, giving regularity $C^{1,\al}$ up to the boundary good points. Finally, choosing variations $\phi N$ with $\phi\in C^\infty_0(\Si),\nabla\phi|_{\Ga}=0$, the same arguments of Proposition 3.1 of \cite{ScBP} can be applied, hence finding that $\supp(\Si)\sm\{\xi_1,...,\xi_P\}$ is $C^\infty$. Let us recall that the regularity arguments are based on the fact that the biharmonic comparison implies the existence of a number $\al\in(0,1)$ such that for any good point $\xi\in\supp\Si$ one has the decay $\int_{\Si\cap B_\ro(\xi)} |\sff_\si|^2\le C\ro^\al$ for any $\ro$ sufficiently small.\\
	
	\noindent By classical arguments one also obtains that $\Si_n\to supp(\Si)$ in Hausdorff distance $d_\cH$ (see also \cite{NoPo20} for a detailed proof). Let us now consider $z\in\Si$ a good point and $x_n,y_n$ points in $\Si_n$ such that $x_n\neq y_n$ for any $n$ and $x_n\to z, y_n\to z$. Fix $\rho>0$ and apply the graphical decomposition method in $B_\rho(z)$. Up to adding two more graphs to the decomposition we can assume that for any $n$ there are smooth functions
	\begin{equation*}
	\begin{split}
		&u_n:\Om_n^u\con L^u_n \to (L^u_n)^\perp, \qquad v_n:\Om_n^v\con L^v_n \to (L^v_n)^\perp, \\
		&x_n\in int(\Om^u_n)\cup\Ga, \qquad\qquad\qquad\,\,\, y_n\in int(\Om^y_n)\cup\Ga,\\
		&graph(u_n)\con \Si_n\cap B_\ro(z), \quad\, graph(v_n)\con \Si_n\cap B_\ro(z).
	\end{split}
	\end{equation*}
	As in \cite{SiEX} there are vectors $\eta_n^u,\eta_n^v\in\R^3$ such that $\eta_n^u\to\eta^u,\eta_n^v\to\eta^v$ and planes $L^u,L^v$ such that $L^u_n\to L^u,L^v_n\to L^v$ and a posteriori smooth functions
	\begin{equation*}
	\begin{split}
		&u:\Om^u\con L^u \to (L^u)^\perp, \qquad v:\Om^v\con L^v \to (L^v)^\perp,
	\end{split}
	\end{equation*}
	so that $z\in L^u\cap L^v\cap  \overline{dom(u)}\cap \overline{dom(v)}$ and
	\begin{equation*}
	\begin{split}
		&graph(u_n)\xrightarrow[]{d_\cH} graph(u), \qquad graph(v_n)\xrightarrow[]{d_\cH} graph(v),
	\end{split}
	\end{equation*}
	\begin{equation*}
	\begin{split}
		\int_{dom(u_n)} |\nabla u_n-\eta^u_n|^2 \le C\sqrt{\al_n}\ro^2, \qquad \int_{dom(v_n)} |\nabla v_n-\eta^v_n|^2 \le C\sqrt{\al_n}\ro^2, \qquad \liminf_n \sqrt{\al_n}\le C\ro^\al,
	\end{split}
	\end{equation*}
	\begin{equation*}
	\int_{dom(u)} |\nabla u-\eta^u|^2 \le C\ro^{2+\al}, \qquad \int_{dom(v)} |\nabla v-\eta^v|^2 \le C\ro^{2+\al},
	\end{equation*}
	with $\al\in(0,1)$. Hence $\nabla u(z)=\eta^u,\nabla v(z)=\eta^v$. Up to reparametrization we can assume $\eta^u=\eta^v=0$ and $L^u_n=L^u,L^v_n=L^v$ for $n$ sufficiently large. If by contradiction we suppose $L^u\neq L^v$, by the convergence of graphs in Hausdorff distance and by smoothness this would imply that $\Si_n$ is not embedded for $n$ big enough, that contradicts our assumptions. This fact is essentially telling us that if two separate sheets of $\Si_n$ contain sequences $x_n,y_n$ converging to tha same good point, then such sheets touch tangentially in the limit. Hence the construction in \cite{La} as used in \cite{ScBP} implies that
	\begin{equation*}
	\forall\ro>0: \quad \exists S_\ro, \overline{\Phi}_\ro:S_\rho\to\R^3\, \mbox{ s.t. }\, \Si\sm \cup_j B_\ro(\xi_j) = \Imm(\overline{\Phi}_\ro),
	\end{equation*}
	where $\{\xi_j\}$ is the set of bad points, $S_\ro$ is a suitable manifold with boundary and $\Imm(\overline{\Phi}_\ro)$ is the varifold induced by $\overline{\Phi}_\ro$.\\
	
	
	\noindent By the above argument and by the fact that as in \cite{SiEX} the monotonicity formula with boundary implies that
	\begin{equation} \label{annular1}
	\forall j\, \forall\ep\,\exists \si_0: \quad\frac{|(x-\xi_j)^\perp|}{|x-\xi_j|}\le\ep \quad \forall x\in\Si\cap B_\si(\xi_j)\sm B_{\si/2}(\xi_j), \,\forall\si\in\cS_j\con(0,\si_0):\,\cL^1((0,\si_0)\sm\cS_j)\le C\ep \si_0^2,
	\end{equation}
	we see that we can construct as in \cite{SiEX} for such $\si$ the annular sets
	\begin{equation} \label{annular2}
	\begin{split}
		&A^\si_j=\{x\in L_j\,\,|\,\, \si/4<|x-\xi_j|<3\si/4 \},\\
		&\cA^\si_j=\{x+z\,\,|\,\,x\in A_j,z\in L_j^\perp, |z|<C_j\si \},
	\end{split}
	\end{equation}
	and smooth functions
	\begin{equation*}
	u_{j,i}:A_j\sm\cup_k d_{ijk}\to L_j^\perp \qquad \mbox{for }i=1,...,N_j,
	\end{equation*}
	giving the graphical decomposition of $\Si\cap \cA_j$.\\
	As in \cite{SiEX} one then finds
	\begin{equation*}
	\exists\si_n\searrow 0 :\quad\Si\sm \cup_j B_{\si_n}(\xi_j) = \overline{\Phi}_n(S), \quad S\simeq \Si_n \sm \cup_j B_{\si_n}(\xi_j),
	\end{equation*}
	for smooth immersions $\overline{\Phi}_n:S\to\R^3$ of a two dimensional manifold $S$ with boundary that is now independent of $n$.\\	
	Hence, using the techniques of \cite{La} as in \cite{ScBP}, one has that $\Si$ is a varifold induced by a branched immersion $\Phi:\Si_{\cg'}\to\R^3$.\\	
	Finally, since $\Si$ is of class $C^\infty$ out of finitely many points, the generalized boundary of $\Si$ is some $\si_\Si=\nu\cH^1\res\Ga$ for a field $\nu$ with $|\nu|=1$ ae. By convergence of the boundary measures $\pa\Si_n\overset{\star}{\rightharpoonup}\pa\Si$, also the generalized boundaries converge: $\si_{\Si_n}=co_n\cH^1\res\Ga\overset{\star}{\rightharpoonup}\si_\Si$. Since $co_n$ converges up to subsequence to some field $W$ weakly in $L^2(\Ga)$, then $G(\Si)=\int_\Ga \lgl \vec{k}_\Ga, -\nu \rgl  \,d\cH^1=\lim_n\int_\Ga \lgl \vec{k}_\Ga, -co_n \rgl  \,d\cH^1=\lim_n G(\Si_n)$ and $W=\nu$.\\
	Hence by Gauss-Bonnet and the usual comparison arguments of \cite{SiEX} one has that the varifold $\Si$ verifies the required minimization property, that is $\cW(\Si)\le\cW(V)$ for any $V\in \cC(\Si_{\cg'})$ for some $\cg'\le\cg$.	
\end{proof}

\begin{remark}
	In the following we will need the fact that in the proof of proposition \ref{simon} one finds (see \cite[(3.55)]{SiEX})
	\begin{equation}\label{debolistarconv}
	\begin{split}
	&|H_n|^2\cH^2\res \Si_n \overset{\star}{\rightharpoonup} |H_\Si|^2\mu_\Si ,\\
	&|\sff_n|^2\cH^2\res \Si_n \overset{\star}{\rightharpoonup} |\sff_\Si|^2\mu_\Si,
	\end{split}
	\end{equation}
	weakly as measures on $\R^3\sm\cup_j \{\xi_j \}$, where $\mu_\Si$ is the weight measure of the limit varifold.
\end{remark}

\textcolor{white}{a}

\subsection{Straight line boundary datum}\label{sub2}

In this section we deal with the minimization of the Willmore energy in the case of complete unbounded surfaces with a straight line as boundary, leading to some tools that we will use in the following.\\
Let $r\con\R^3$ be a straight line and fix an integer $\cg\ge 1$. We now consider in detail the minimization problem $\min \{\cW(\Si)| \Si\in\cA \}$.
We can state and prove the first main result of the subsection.

\begin{prop} \label{prop1}
	If $\cg\ge 1$, the minimization problem \eqref{ref21} has no solution. Hence the same holds for the minimization problems \eqref{problemr} and \eqref{probr}.
\end{prop}

\begin{proof}
	Suppose by contradiction that $\Si$ is a minimizer of problem \eqref{ref21}. It is known (see for example \cite{Pa}) that a necessary boundary condition is then
	$$H=\lgl \vec{k}, N \rgl \qquad\mbox{on }\S^1,$$
	where $\vec{k}$ is the curvature vector $\vec{k}(p)=-p$ of $\S^1$ and $N$ is the unit normal field orienting $\Si$. This is equivalent to
	\begin{equation}\label{ref23}
	\sff(co,co)\equiv0 \qquad\mbox{on }\S^1,
	\end{equation}
	where $\sff(v,w)=-\lgl \pa_v N, w \rgl$ for any $v,w\in T\Si$ is the second fundamental form of $\Si$ and $co$ is the unit outward conormal of $\Si$ (see for example \cite{AbToCS}, page 190). Up to rotation let $v=(-1,0,0)$ of multiplicity $1$ in $\Si$. By the same arguments and using the notation of Lemma \ref{cop} we can construct a conformal map $F:U\to F(U)$ with $U$ open with $\Si\con U$ such that
	\begin{equation} \label{conff}
	\begin{split}
		&i)\,\,F(p)=\begin{cases}
			I_{1,v}^{-1}\circ T_{-v/2}\circ \bR\circ D_\al \circ T_{v/2} \circ I_{1,v} (p)  & p\in U\sm\{ v\},\\
			v & p=v,
		\end{cases},\\
		&ii)\,\, d_{C^1}\bigg(\bR\circ D_\al\circ T_{v/2} \circ I_{1,v} (\Si)\sm B_\ga(0), \Pi  \sm B_\ga(0) \bigg)<\be,\\
		&iii)\,\, co_{\Pi} \equiv (0,0,-1),\\
		&iv)\,\, 0\not\in T_{-v/2}\circ \bR\circ D_\al \circ T_{v/2} \circ I_{1,v} (\Si),
	\end{split}
	\end{equation}
	for arbitrarily small $\be,\ga\in(0,1)$ and corresponding sufficiently small $\al=\al(\be,\ga)\in(0,1)$, where $\Pi$ is a half plane passing through the origin with $\pa\Pi=\{x=z=0\}$, and $\,co_{\Pi}$ is conormal vector of $\Pi$. As usual, condition $ii)$ means that $\bR\circ D_\al\circ T_{v/2} \circ I_{1,v} (\Si)\sm B_\ga(0)$ is a graph of a function $f$ over $\Pi\sm B_\ga(0)$ with $C^1$ norm smaller than $\be$. Since $D(I_{1,v}(\Si))<+\infty$, along almost every radial direct $t\hat e$ on $\Pi$, i.e. $\hat e \in \Pi$ and $|\hat e|=1$, one has that $(\nabla^2 f)(t\hat e) \to 0$ at $t\to\infty$. For a such direction consider the curve $\ga_{\hat e}(t)=(F\circ f)(t\hat e)$ contained in $F(\Si)$; such curve tends to $v$ as $t\to\infty$ and forms an angle $\te_{\hat e}$ with $\{z=0\}$, and the choice of $\hat e$ implies that, up to reparametrization and for large times, the curve is close in $C^2$ norm to the geodesic in $S^2\cap B_r(v)\cap \{z\ge0\}$ passing from $v$ forming the same angle with $\{z=0\}$ at $v$. Therefore for $\hat e$ sufficiently close to $(0,0,1)$ we have that $[\sff_{F(\Si)}(\tau_{\ga_{\hat e}}(t),\tau_{\ga_{\hat e}}(t))](v) \ge \tfrac12 \sff_{S^2}(co_{F(\Si)},co_{F(\Si)})>0$ for $t$ large, where $\tau_{\ga_{\hat e}}$ is the tangent vector of $\ga_{\hat e}$. Since we can choose $\hat e$ arbitrarily close to $(0,0,1)$, by continuity we see that $[\sff_{F(\Si)}(co_{F(\Si)},co_{F(\Si)})](v)>0$. However since $F$ is conformal we have $(\cW+G)(F(\Si))=(\cW+G)(\Si)$, and then $F(\Si)$ is a minimizer too and has to satisfy \eqref{ref23}, which gives a contradiction.\\
	By Corollary \ref{corequiv}, problem \eqref{problemr} does not admit a minimizer.
\end{proof}
%

\begin{prop} \label{ref27}
	Let $\Si_n$ be a minimizing sequence of embedded surfaces for the problem \eqref{probr}. Then, up to subsequence and up to rescaling, $\Si_n$ converges in the sense of varifolds to a multiplicity $1$ half plane $\Pi$ with boundary $r$ on every ball $B_R(0)\con\R^3$.
\end{prop}

\begin{proof}
	Let us prove the convergence of $\Si_n$ as varifolds in any open ball $B_R(0)$. For any $n$ the Gauss-Bonnet like identity for asymptotically flat manifolds reads
	\begin{equation*}
	\cW(\Si_n)=\frac{1}{4}D(\Si_n)+\pi\chi(\Si_n)+\pi,
	\end{equation*}
	where $\chi(\Si_n)=2-2\cg-1$ is the Euler characteristic of $\Si_n$. Such relation can be obtained by approximation letting $\ro\to\infty$ in the classical Gauss-Bonnet equation for the compact surfaces $\Si_n\cap B_\ro(0)$ keeping $n$ fixed. By minimality of the sequence we get a uniform bound on $D(\Si_n)$. Moreover $\pa\Si_n\cap B_R(0)=r\cap B_R(0)$ has finite constant length, depending only on $R$.\\
	By the definition of asymptotically flat surface, for any $n$ there is $R_n>0$ be such that $\pa B_{R_n}(0)\cap \Si_n$ has length bounded by $2\pi R_n$ and $\Si_n\sm B_{R_n}(0)$ is the graph of a function $f_n$ with $\|f_n\|_{C^{1}}<\tfrac1n$. Consider the sequence $\Si_n':=\frac{R}{R_n}\Si_n$, so that $\cH^1(\pa B_R(0)\cap \frac{R}{R_n}\Si_n)\le2\pi R$ is uniformly bounded in $n$.\\
	In order to get the desired varifold convergence, we only need to estimate $|\Si'_n\cap B_R(0)|$. Call $\Si''_n:=\frac{1}{R}\Si'_n$, then
	\begin{equation*}
	\begin{split}
		2|\Si''_n\cap B_1(0)| &= \int_{\Si''_n\cap B_1(0)} div_{\Si''_n} x =-2\int_{\Si''_n\cap B_1(0)} \lgl \vec{H}''_n,x\rgl +\int_{\pa(\Si''_n\cap B_1(0))} \lgl x, co''_n(x)\rgl \,d\cH^1=\\&= |\Si''_n\cap B_1(0)| -\int_{\Si''_n\cap B_1(0)} (1-|x^\perp|^2)-\int_{\Si''_n\cap B_1(0)} |\vec{H}''_n+x^\perp|^2+\\&\quad+\cW(\Si''_n\cap B_1(0))+\int_{\pa (\Si''_n\cap B_1(0))} \lgl x, co''_n(x)\rgl \,d\cH^1 \le \\
		&\le |\Si''_n\cap B_1(0)|+\cW(\Si'_n)+\cH^1(\pa (\Si''_n\cap B_1(0)))\le\\
		&\le |\Si''_n\cap B_1(0)|+\cW(\Si'_n)+C(R),
	\end{split}
	\end{equation*}
	hence $|\Si'_n\cap B_R(0)|=R^2|\Si''_n\cap B_1(0)|$ is uniformly bounded in $n$, depending on $R$. Therefore we got that $\Si_n'$ converges as varifolds on the ball $B_R(0)$, up to passing to a subsequence. By construction, the condition $\|f_n\|_{C^1}<\tfrac1n$ implies that the limit varifold is a density $1$ half plane.
\end{proof}

\begin{lemma} \label{convconormali}
	Let $\Si_n$ be a minimizing sequence of embedded surfaces for the problem \eqref{probr} and assume $\Si_n\to\Pi$ in the sense of varifolds, where $\Pi$ is a half plane. Denote by $co_n$ and $co_\Pi$ respectively the conormal fields of $\Si_n$ and $\Pi$. Then for every $R>0$ it holds
	\begin{equation*}
	co_n\to co_\Pi \qquad \mbox{strongly in }L^2(r\cap B_R(0)).
	\end{equation*}
\end{lemma}

\begin{proof}
	Let us rewrite the monotonicity formula with boundary as
	\begin{equation} \label{reff11}
	\begin{split}
		\frac{|T_\si|}{\si^2}-\frac{|T_\ro|}{\ro^2}+\int_{T_{\ro,\si}} \frac{|(p-p_0)^\perp|^2}{|p-p_0|^4}=&\int_{T_\ro} \frac{1}{\ro^2}\lgl \vec{H}, (p-p_0)^\perp\rgl -\int_{T_\si} \frac{1}{\si^2} \lgl \vec{H}, (p-p_0)^\perp\rgl -\int_{T_{\ro,\si}} \bigg\lgl \vec{H}, \frac{(p-p_0)^\perp}{|p-p_0|^2} \bigg\rgl+\\
		&+ \frac{1}{2}\int_{r_\ro} \frac{1}{\ro^2} \lgl co, p-p_0\rgl +\frac{1}{2}\int_{r_{\ro,\si}} \bigg\lgl co, \frac{p-p_0}{|p-p_0|^2} \bigg\rgl +\frac{1}{2}\int_{r_\si} \frac{1}{\si^2} \lgl co, p-p_0\rgl,
	\end{split}
	\end{equation}
	valid for $T=\Si_n$ for any $n$, and $0<\si<\ro<R/2$, $p_0\in r\cap B_{R/2}(0)$, where for any set $A$ we have used $A_\si=A\cap B_\si(p_0)$ and $A_{\ro,\si}=A_\ro\sm A_\si$. By orthogonality, the scalar products in the last three integrals of \eqref{reff11} vanish. Moreover for any $\de>0$ we can estimate
	$$ -\int_{T_{\ro,\si}} \bigg\lgl \vec{H}, \frac{(p-p_0)^\perp}{|p-p_0|^2} \bigg\rgl \le \frac{1}{2\de} \int_{T_{\ro,\si}} |H|^2 +\frac{\de}{2}\int_{T_{\ro,\si}} \frac{|(p-p_0)^\perp|^2}{|p-p_0|^4}. $$
	Absorbing the last integral on the left in \eqref{reff11} and neglecting the resulting positive term, we estimate
	$$	|(\Si_n)_\si|\le \frac{\si^2}{\ro^2}|(\Si_n)_\ro|+M(R)\si^2,$$
	for a positive constant $M(R)$ independent of $n,\si,p_0$ and depending only on $R$. Being $|\Si_n\cap B_R(0)|$ uniformly bounded in $n$ by the varifold convergence, we get
	\begin{equation} \label{reff12}
	|(\Si_n)_\si|\le M(R)\si^2,
	\end{equation}
	for another positive constant $M(R)$ independent of $n,\si,p_0$ and depending only on $R$.\\
	Now fix a vector field $X\in C^1_c(B_R(0);\R^3)$. Also choose $h>0$ small and a cylindric cut off function $\chi\in C^\infty(\R^3)$ such that $\chi|_r\equiv1$ and $\chi(p)=0$ if $d(p,r)\ge h$. By varifold convergence toward a half space $\Pi$, we have
	\begin{equation}\label{A}
	\lim_n \bigg( \int_{\Si_n\cap B_R(0)} \lgl \vec{H}_n,\chi X\rgl + \int_{\Si_n\cap B_R(0)} \lgl \vec{H}_n,(1-\chi) X\rgl +\int_{r\cap B_R(0)} \lgl X,co_n\rgl \bigg)= \int_{r\cap B_R(0)} \lgl X,co_\Pi\rgl .
	\end{equation}
	Also
	\begin{equation}\label{B}
	\lim_n \int_{\Si_n\cap B_R(0)} \lgl \vec{H}_n,(1-\chi) X\rgl =0 .
	\end{equation}
	Moreover
	$$\bigg| \int_{\Si_n\cap B_R(0)} \lgl \vec{H}_n,\chi X\rgl \bigg|\le \cW(\Si_n)^{1/2}\|X\|_\infty |\Si_n\cap\{\chi\ge0 \}\cap B_R(0)|^{1/2} .  $$
	For any $h$ small there is $\si>0$ such that $(\{\chi\ge0 \}\cap B_R(0))\con \bcup_{i=1}^{\frac{R}{\si}+1} B_\si(x_i) $ for suitable $x_i\in r\cap B_R(0)$, then by \eqref{reff12} we get
	\begin{equation}\label{C}
	\lim_n \bigg| \int_{\Si_n} \lgl \vec{H}_n,\chi X\rgl \bigg| \le \cW(\Si_n)\|X\|_\infty \sqrt{M} \sqrt{R+\si}\sqrt{\si},
	\end{equation}
	with $M=M(R)$ independent of $n,\si,\chi$.\\ 
	Up to subsequence $co_n\wto W$ weakly in $L^2(r\cap B_R(0))$, so there exists the limit $\lim_n \int_{r\cap B_R(0)} \lgl X,co_n\rgl= \int_{r\cap B_R(0)} \lgl X,W\rgl$; then by \eqref{A} there exists also the limit $\lim_n \int_{\Si_n\cap B_R(0)} \lgl \vec{H}_n,\chi X\rgl$. Sending $\si\to0$, and thus $h\to0$, we get
	\begin{equation*}
	\begin{split}
		\lim_{\si\searrow0}\lim_n &\bigg( \int_{\Si_n\cap B_R(0)} \lgl \vec{H}_n,\chi X\rgl + \int_{\Si_n\cap B_R(0)} \lgl \vec{H}_n,(1-\chi) X\rgl +\int_{r\cap B_R(0)} \lgl X,co_n\rgl \bigg)  = \lim_n \int_{r\cap B_R(0)} \lgl X,co_n\rgl=\\
		&=  \int_{r\cap B_R(0)} \lgl X,co_\Pi\rgl =  \int_{r\cap B_R(0)} \lgl X,W\rgl .
	\end{split}
	\end{equation*}
	This is true for any $X\in  C^1_c(B_R(0);\R^3)$, then $W=co_\Pi$. Since $co_n\wto co_\Pi$ weakly in $L^2(r\cap B_R(0))$ and $\|co_n\|_{L^2(r\cap B_R(0))} = \|co_\Pi\|_{L^2(r\cap B_R(0))} $, then the convergence also holds strongly in $L^2(r\cap B_R(0))$.\\
	Since this is true for any subsequence of $co_n$, the convergence $co_n\to co_\Pi$ also holds for the original sequence $co_n$.
\end{proof}

\begin{cor} \label{iinf}
	If $\cg\ge1$, the infimum of problem \eqref{probr} is equal to $\be_\cg-4\pi$.
\end{cor}

\begin{proof}
	Let $\Si_n$ be a minimizing sequence of embedded surfaces for the problem \eqref{probr} and by Proposition \ref{ref27} assume $\Si_n\to\Pi$ in the sense of varifolds, where $\Pi$ is a half plane. Up to small modifications we can assume $0\not\in\Si_n$ for any $n$. By Lemma \ref{convconormali} the conormal fields $co_n$ of $\Si_n$ converge in $L^2$ norm on compact subsets of the line $r$ to $co_\Pi$. Up to rotation and translation we can suppose that $co_\Pi\equiv (-1,0,0)$ and $r=\{(1/2,t,0)|t\in\R \}$. Let us perform the transformation
	\begin{equation*}
	I^{-1}_{1,v}(p)= \frac{p}{|p|^2}+(-1,0,0) \qquad\forall p\neq 0.
	\end{equation*}
	By Lemma \ref{sphericalr} we get surfaces $\Si_n':=\{(-1,0,0)\}\cup F(\Si_n)$ with $\pa \Si_n'=\S^1_1(0)$ such that
	\begin{equation} \label{ref24}
	\cW(\Si_n)=(\cW+G)(\Si'_n)-2\pi\ge \be_\cg -4\pi-2\pi + G(\Si'_n),
	\end{equation}
	where the last inequality follows by Theorem \ref{nonexS1}.\\
	The conormal $co'_n$ of $\Si'_n$ at a point $(\cos \te,\sen \te,0)=I^{-1}_{1,v}\big(\frac{1}{2},\frac{\sen \te}{2(1+\cos \te)},0\big)$ is
	\begin{equation*}
	\begin{split}
		(co'_n)_{(\cos \te,\sen \te,0)}=\frac{dI^{-1}_{1,v}(co_n)}{|dI^{-1}_{1,v}(co_n)|}\bigg|_{\big(\frac{1}{2},\frac{\sen \te}{2(1+\cos \te)},0\big)}= \bigg( co_n -2 \frac{\lgl q, co_n\rgl }{|q|^2}q \bigg)\bigg|_{q=\big(\frac{1}{2},\frac{\sen \te}{2(1+\cos \te)},0\big)}.
	\end{split}
	\end{equation*}
	The direct calculation then shows that if $co_n\to (-1,0,0)$ in $L^2\big(\big\{\big(\frac{1}{2},\frac{\sen \te}{2(1+\cos \te)},0\big) | \te \in (\pi-\al,\pi+\al)\big\}\big)$ for some $\al\in(0,\pi)$, then $co'_n \to \hat{p}$ in $L^2(\{ (\cos \te,\sen\te,0)|\te \in (\pi-\al,\pi+\al) \})$, where $\hat{p}$ is the vector field $\hat{p}_{(\cos\te,\sen\te,0)}=(\cos\te,\sen\te,0)$. By Lemma \ref{convconormali} such convergence happens for any $\al\in(0,\pi)$, then
	$$G(\Si'_n)=\int_{\S^1} \lgl p, (co'_n)_p\rgl \,d\cH^1(p)\to 2\pi.$$
	Hence passing to the limit in \eqref{ref24} gives the conclusion.
\end{proof}

\noindent Putting together Proposition \ref{prop1} with Corollary \ref{iinf}, we get the following

\begin{thm} \label{thmr}
	If $\cg\ge1$, problem \eqref{probr}, and equivalently problem \eqref{problemr}, has no solution and the infimum equals $\be_\cg-4\pi$.\\
\end{thm}

\textcolor{white}{text}

\subsection{Boundary data admitting minimizers}

\noindent In this part we prove the following result.

\begin{thm} \label{thmmain}
	\textcolor{white}{text}
	\begin{enumerate}[label={\normalfont(\roman*)}]
	\item If $\cg=1$, $\inf_{\cC(\Si_\cg)}\cW= \inf_{\cC_{imm}(\Si_\cg)} \cW$, and problem \eqref{problemmin} has no solution, then the infimum of the problem equals $\be_1-4\pi=2\pi^2-4\pi$.
	\item Let $\cg=1$. There exist infinitely many closed convex planar smooth curves $\Ga$ such that if $\inf_{\cC(\Si_\cg)}\cW= \inf_{\cC_{imm}(\Si_\cg)} \cW$, then problem \eqref{problemmin} has minimizers.
	\item If $\cg\ge2$, $\inf_{\cC(\Si_{\cg'})}\cW= \inf_{\cC_{imm}(\Si_{\cg'})} \cW$ and problem \eqref{problemmin} has no solution for any genus $1\le\cg'\le\cg$, then the infimum of the problem equals $\be_\cg-4\pi$.
	\end{enumerate}
\end{thm}


\begin{remark} \label{remg1}
	In the assumptions of Theorem \ref{thmmain}, denote by $\Si_n\in\cC^*(\Si_1)$ a minimizing sequence for problem \eqref{problemmin} with $\cg=1$, let $\Si$ be a varifold limit given by Proposition \ref{simon}, and call $D$ the planar region enclosed by $\Ga$. If problem \eqref{problemmin} has no solution, by Proposition \ref{simon} the limit varifold $\Si$ is induced by a branched immersion of the disk and minimizes the problem for genus $\cg'=0$, hence $\Si=D$. Also by \eqref{diffeom} there exists a unique bad point $\xi$ such that it absorbs the topology of $\Si_n$, more precisely
	\begin{equation} \label{topg1}
	\exists\si_n\searrow 0 :\quad D\sm B_{\si_n}(\xi)\simeq \Si_n \sm  B_{\si_n}(\xi).
	\end{equation}
\end{remark}

\medskip

\noindent In the next lemmas we find lower bounds that will prove point $i)$ of Theorem \ref{thmmain}. To this aim we will always assume that we are in the setting of Remark \ref{remg1} and we will divide the cases of bad point $\xi$ lying on the boundary $\Ga$ or not. We first examine the case in which $\xi$ is assumed to be in the interior of $D$. The strategy goes as follows: we first isolate the bad point and the topology of the minimizing sequence as stated in Remark \ref{remg1}, also by identifying an annular region such that the surface around the bad point can be written as a graph of functions converging to zero in $C^1$, and then we substitute the remaining part of the sequence with graphs of biharmonic functions converging to zero in $C^2$. In the case of $\xi\in \Ga$ the strategy will be exactly the same, except that we will also need to take into account some technicalities related to the presence of the boundary.

\begin{lemma}\label{xiinterno}
	Let $\Ga$ be a closed compact smooth simple planar curve and denote by $\Si_n\in\cC^*(\Si_1)$ a minimizing sequence for problem \eqref{problemmin} with $\cg=1$, let $\Si$ be a varifold limit given by Proposition \ref{simon}, and call $D$ the planar region enclosed by $\Ga$. Suppose that there is no solution to the minimization problem \eqref{problemmin} in the case of genus $\cg=1$ and that $\inf_{\cC(\Si_\cg)}\cW= \inf_{\cC_{imm}(\Si_\cg)} \cW$, then call $\xi$ and $\si_n$ the bad point and the sequence of Remark \ref{remg1}.\\
	If the bad point $\xi$ lies in the interior of $D$, then the infimum of problem \eqref{problemmin} is equal to $\be_1-4\pi=2\pi^2-4\pi$.
\end{lemma}

\begin{proof}
	By \eqref{annular1}, \eqref{annular2}, by the convergence in Hausdorff distance, and by the Graph Decomposition Lemma (Appendix B), we can fix $\si_0>0$ such that $B_{2\si_0}(\xi)\cap \Ga=\emptyset$ and such that for any $\si\in\cS\con(0,\si_0)$ we have an annular region $A^\si=\{ x\in D| \si/4<|x-\xi|<3\si/4 \}$ and functions $u_n^\si:A^\si\sm \cup_k d_{nk}^\si \to D^\perp$ with
	\begin{equation*}
	\begin{split}
		& \sup_{dom(u_n^\si)} \frac{|u_n^\si|}{\si}+|\nabla u_n^\si|\le C\ep^{1/6},\\
		&(\Si_n\cap \cA^\si)^* = graph(u_n^\si)\cup_k P_{kn}^\si,\\
		& \pi_{D}(P_{kn}^\si)=d_{kn}^\si, \quad \sum_k diam(P_{kn}^\si)\le C\ep^{1/6}\si,\\
		&  d_{kn}^\si\cap \pa dom(u_n^\si)=\emptyset,\\
		& \cL^1(\cS(\si_0))\ge \si_0-c\ep\si_0,
	\end{split}
	\end{equation*}
	where $\cA^\si=\{ x+z| x\in A^\si,z\in D^\perp, |z|<\si/4 \}$ and in $(\Si_n\cap \cA^\si)^*$ the star denotes the selection of the connected component of $(\Si_n\cap \cA^\si)$ such that
	\begin{equation}\label{eq:Homeomorphism}
	(\Si_n\cap \cA^\si)^*\cup K \cong \Si_n\sm\Ga,
	\end{equation}
	where $K$ is the unique connected component of $\Si_n\cap B_{\frac{\si}{4}}(\xi)$ satisfying the homeomorphism in \eqref{eq:Homeomorphism} (roughly speaking, we are selecting the right annulus of $\Si_n$ around the ball accumulating the topology).\\
	From now on let $\si\in \cS(\si_0)\cap (\si_0/2,\si_0)$ be fixed (which exists for $\ep$ small enough); by the properties above we can fix $\bar{\si}\in\bar{S}(\si)\con (\si/4,3\si/4)$ with $\cL^1(\bar{S}(\si))\ge \si/4$ such that
	\begin{equation*}
	\pa B_{\bar{\si}}(\xi)\cap D\con dom(u_n^\si) \quad\mbox{for infinitely many }n,
	\end{equation*}
	and
	\begin{equation*}
	\int_{graph(u_n^\si|_{\pa B^{\R^2}_{\bar{\si}}(\xi)})} |\sff_n|^2 \le \frac{2}{\cL^1(\bar{S}(\si))}\int_{graph (u_n^\si)} |\sff_n|^2\le \frac{8}{\si}\int_{\Si_n\cap T_{\bar{\si}}} |\sff_n|^2 \quad\mbox{for infinitely many }n,
	\end{equation*}
	where $T_{\bar{\si}}$ is a fixed tubolar neighborhood of $graph(u_n^\si|_{\pa B^{\R^2}_{\bar{\si}}(\xi)})$ independent of $n$ such that no bad points belong to $\overline{T_{\bar{\si}}}$.\\
	By the biharmonic comparison Lemma 2.2 in \cite{SiEX} (Appendix B) and the minimality assumption on the sequence, we can estimate
	\begin{equation} \label{reff4}
	\begin{split}
		\int_{\pa B^{\R^2}_{\bar{\si}}(\xi)} |\nabla^2 u^\si_n|^2 \le C(\ep)\int_{graph(u_n^\si|_{\pa B^{\R^2}_{\bar{\si}}(\xi)})}|\sff_n|^2\le C(\ep)\frac{8}{\si} \int_{\Si_n\cap T_{\bar{\si}}} |\sff_n|^2,
	\end{split}
	\end{equation}
	and for any $x,y\in \pa B^{\R^2}_{\bar{\si}}(\xi)$, denoting by $\ga(x,y)$ the shortest arc in $\pa B^{\R^2}_{\bar{\si}}(\xi)$ from $x$ to $y$ and by $T$ a tangent unit vector field to $\pa B^{\R^2}_{\bar{\si}}(\xi)$ we have
	\begin{equation*}
	\begin{split}
		|\nabla u^\si_n(y)-\nabla u^\si_n(x)|&\le \int_{\ga(x,y)} \bigg|\frac{\pa}{\pa T} \nabla u^\si_n  \bigg|\,d\cH^1 \le \int_{\ga(x,y)} |\nabla^2 u^\si_n| \,d\cH^1 \le\\
		&\le \bigg(\int_{\pa B^{\R^2}_{\bar{\si}}(\xi)} |\nabla^2 u^\si_n|^2\bigg)^{1/2} dist_{\pa B^{\R^2}_{\bar{\si}}(\xi)}(x,y)^{1/2}.
	\end{split}
	\end{equation*}
	Being also $|\nabla u^\si_n|\le C\ep^{1/6}$, up to subsequence by Ascoli-Arzel\`{a} we have
	\begin{equation*}
	u^\si_n \xrightarrow[n]{}0,\,\,\nabla u^\si_n \xrightarrow[n]{}0 \qquad \mbox{uniformly on } \pa B^{\R^2}_{\bar{\si}}(\xi).
	\end{equation*}
	Now let $R=A\bar{\si}>0$ such that $\Ga\con B_{R/2}(\xi)$, and call $B:=B^{\R^2}_{R}(\xi)\sm  B^{\R^2}_{\bar{\si}}(\xi)$. There exist well defined functions $w_n\in C^\infty(\overline{B})$ such that
	\begin{equation*}
	\begin{cases}
		\De^2w_n=0 & \mbox{on } B,\\
		w_n=u^\si_n & \mbox{on } \pa B^{\R^2}_{\bar{\si}}(\xi),\\
		\nabla w_n=\nabla u^\si_n  & \mbox{on } \pa B^{\R^2}_{\bar{\si}}(\xi),\\
		w_n=0   & \mbox{on } \pa B^{\R^2}_{R}(\xi) ,\\
		\nabla w_n=0 & \mbox{on } \pa B^{\R^2}_{R}(\xi).\\
	\end{cases}
	\end{equation*}
	It is readily checked that $w_n$ minimizes $\int_B |\nabla^2 v|^2$ among all $v$ with the same boundary data. Passing to the infimum on such $v$ in the inequality $\int_B  |\nabla^2 w_n|^2\le \int_B  |\nabla v|^2 + \int_B  |\nabla^2 v|^2$ we get
	\begin{equation*}
	\begin{split}
		\int_B  |\nabla^2 w_n|^2\le \| (\nabla u^\si_n)|_{\pa B^{\R^2}_{\bar{\si}}(\xi)}\|^2_{H^{1/2}(\pa B^{\R^2}_{\bar{\si}}(\xi))}\le \| u^\si_n\|_{H^1(\pa B^{\R^2}_{\bar{\si}}(\xi))}^2+ \| \nabla u^\si_n\|^2_{L^2(\pa B^{\R^2}_{\bar{\si}}(\xi))}+\|\nabla^2 u^\si_n\|_{L^2(\pa B^{\R^2}_{\bar{\si}}(\xi))}^2.
	\end{split}
	\end{equation*}
	Applying the same inequalities to the function $w_n-l_n$ for suitable affine functions $l_n$, we estimate
	\begin{equation*}
	\begin{split}
		\int_B  |\nabla^2 w_n|^2\le C(B) \|\nabla^2  u^\si_n\|_{L^2(\pa B^{\R^2}_{\bar{\si}}(\xi))}^2= \bar{\si}C^* \|\nabla^2  u^\si_n\|_{L^2(\pa B^{\R^2}_{\bar{\si}}(\xi))}^2,
	\end{split}
	\end{equation*}	
	where $C^*$ is a universal positive constant. Hence by \eqref{reff4} we obtain
	\begin{equation*}
	\int_B  |\nabla^2 w_n|^2\	\le \bar{\si}C^* C(\ep)\frac{8}{\si} \int_{\Si_n\cap T_{\bar{\si}}} |\sff_n|^2 \le 6C^* C(\ep) \int_{\Si_n\cap T_{\bar{\si}}} |\sff_n|^2,
	\end{equation*}
	then by \eqref{debolistarconv} we get
	\begin{equation} \label{reff5}
	\begin{split}
		\limsup_n \int_B |\nabla^2 w_n|^2 &\le 6C^* C(\ep)\limsup_n \bigg(|\sff_n|^2\cH^2\res\Si_n \bigg)\big(\overline{T_{\bar{\si}}} \big)\le\\
		&\le 6C^* C(\ep) \bigg(|\sff_D|^2\cH^2\res D \bigg)\big(\overline{T_{\bar{\si}}} \big) =0.
	\end{split}
	\end{equation}
	By Remark \ref{remg1} we conclude that the $C^{1,1}$ surface $\tilde{\Si}_n$ given by extending $w_n$ to zero on $(B^{\R^2}_R(0))^c$ and then gluing $graph(w_n)$ with $\Si_n$ at the curve $u^\si_n(\pa B^{\R^2}_{\bar{\si}}(\xi))$ is an asymptotically flat surface of genus $1$; hence its Willmore energy is no greater or equal than $2\pi^2-4\pi$. By \eqref{reff5} we conclude that
	\begin{equation} \label{ref43}
	2\pi^2-4\pi\le \liminf_n \cW(\tilde{\Si}_n)\le \liminf_n \cW(\Si_n) + \cW(graph(w_n))= \lim_n \cW(\Si_n)\le 2\pi^2-4\pi.
	\end{equation}
\end{proof}

\begin{remark}
	Let us mention that it is possible to give a better description of the behavior of the bad point $\xi$ in the interior of $D$ as follows. In the assumptions and notation of Lemma \ref{xiinterno} and of its proof, there exists a sequence of blow ups $A_n(\tilde{\Si}_n-\xi)$ with $A_n\to\infty$ that converges (up to subsequences) to $P$ up to subsequence in the sense of varifolds, where $P$ is an asymptotically flat surface of genus $1$ with one end with minimal energy equal to $e_1$. This easily follows by choosing $A_n$ suitably and by the fact that the Willmore energy is scaling invariant and $\cW(\tilde{\Si}_n)\to e_1$.
\end{remark}

\noindent Now we consider the case in which $\xi$ lies on the boundary curve. In complete analogy with the proof of Lemma \ref{xiinterno}, this time we will compare the energy of the minimizing sequence with the energy of an asymptotically flat surface of genus $1$ having a straight line as boundary, and thus exploit the results in Subsection \ref{sub2}. To this aim, we will first need to flatten the curve $\Ga$ around the bad point, and this will be the output of the technical part of the next lemma; the rest of the argument is then analogous to the strategy employed in the case of $\xi$ belonging to the interior of $D$.

\begin{lemma}\label{xibordo}
	Let $\Ga$ be a closed compact smooth simple planar curve and denote by $\Si_n\in\cC^*(\Si_1)$ a minimizing sequence for problem \eqref{problemmin} with $\cg=1$, let $\Si$ be a varifold limit given by Proposition \ref{simon}, and call $D$ the planar region enclosed by $\Ga$. Suppose that there is no solution to the minimization problem \eqref{problemmin} in the case of genus $\cg=1$ and that $\inf_{\cC(\Si_\cg)}\cW= \inf_{\cC_{imm}(\Si_\cg)} \cW$, then call $\xi$ and $\si_n$ the bad point and the sequence of Remark \ref{remg1}.\\
	If there is no solution to the minimization problem \eqref{problemmin} in the case of genus $\cg=1$, if the bad point $\xi$ lies on the boundary $\Ga$, then the infimum of problem \eqref{problemmin} is equal to $\be_1-4\pi=2\pi^2-4\pi$.
\end{lemma}

\begin{proof}
	\noindent Applying the same arguments and estimates of the proof of Lemma \ref{xiinterno} on a sequence of radii $(\si_0)_n\searrow0$ in place of $\si_0$, we deduce, up to passing to a diagonal sequence, the existence of radii $r_n\searrow0$ and functions $u_n:\Om_n\con D\to D^\perp$ such that
	\begin{equation} \label{inizio'}
	\begin{split}
		&\pa B_{r_n}^{\R^2}(\xi)\cap D \con \Om_n,\\
		& graph(u_n)\con \Si_n,\\
		& \Si_n\sm ( graph(u_n|_{\pa B_{r_n}^{\R^2}(\xi)\cap D })\cup \Ga) \simeq D\sqcup \Si_n\sm\Ga,\\
		& \int_{\pa B_{r_n}^{\R^2}(\xi)\cap D } |\nabla^2 u_n|^2 \le \frac{8}{\tilde{r}_n} \int_{\Si_n\cap T_{r_n}} |\sff_n|^2 \le \frac{1}{n},\\
		& |u_n|,|\nabla u_n|\le r_n^2 \quad\mbox{ on 
		} \pa B_{r_n}^{\R^2}(\xi)\cap D ,\\
		&\sup_{\Om_n} |\nabla u_n|\le C\ep^{1/6},
	\end{split}
	\end{equation}
	where each $T_{r_n}$ is a closed tubolar neighborhood of $graph(u_n|_{\pa B_{r_n}^{\R^2}(\xi)\cap D})$ not containing bad points and $\tilde{r}_n$ is a suitable radius (here $\tilde{r}_n,r_n$ respectively do the job of $\si,\bar{\si}$ in the proof of Lemma \ref{xiinterno}).\\
	Up to translation and rotation assume $\xi=0$, $\Ga\con \{z=0 \}$, $\{x=0\}$ is tangent to $\Ga$ at $\xi=0$ and $(0,-1,0)$ points outside of $D$. In this way for $\ro\le\ro_0(\Ga)\le1$ let $\Ga_\ro:I_\ro\con\R\to\R$ be the function such that $\Ga\cap B_\ro(0)=(graph(\Ga_\ro),0)$. Also for any $\ro\in(0,\ro_0(\Ga))$ extend $\Ga_\ro$ on the whole line arbitrarily but assuming that
	\begin{equation} \label{ref34}
	\begin{split}
		&\Ga_\ro:\R\to\R, \quad\Ga_\ro\in C^\infty_c(-10,10), \\
		& \Ga\cap B_\ro(0)=(graph(\Ga_\ro),0)\cap B_\ro(0),\\
		& \sup_\R |\Ga_\ro'|\le \sup_{(-\ro,\ro)} |\Ga_\ro'|,\\
		& \sup_\R |\Ga_\ro''|\le \sup_{(-\ro,\ro)} |\Ga_\ro''|.
	\end{split}
	\end{equation}
	Now for indexes $n$ and radii $r_n\le\ro_0/2$ as in \eqref{inizio'}, let $\tilde{\Si}_n:=\frac{1}{r_n}\Si_n$ and let $f_{n}$ be the diffeomorphism
	\begin{equation*}
	f_{n}:\R^3\to\R^3,\quad f_{n}(p)= p -\bigg( 0, \frac{1}{r_n}\Ga_{r_n}(r_n\pi_1(p)),0 \bigg),
	\end{equation*}
	where $\pi_1(p)=x$ if $p=(x,y,z)$. Then call $\Si'_n:=f_n(\tilde{\Si}_n)$. In this way we have deformed a neighborhood of $0$ in $\Ga$ into a segment; more precisely: $\frac{1}{r_n}(I_{r_n},0,0)\con\pa \Si'_n$. Since
	\begin{equation*}
	d(f_{n})_p=\begin{pmatrix}
		1 & 0 & 0\\
		-\Ga_{r_n}'\big|_{r_n\pi_1(p)} & 1 & 0 \\
		0 & 0 & 1\\
	\end{pmatrix}, \qquad \pa_i\pa_j(f_{n})(p)=-r_n \de_i^x\de_j^x \begin{pmatrix}
		0\\ \Ga_{r_n}''\big|_{r_n\pi_1(p)} \\ 0
	\end{pmatrix},
	\end{equation*}
	and $\Ga_\ro'\xrightarrow[\ro\to0]{}0$ uniformly on $\R$ and $|\Ga_\ro''|\le k_0$ with $k_0$ depending only on the curvature of $\Ga$, we then get that 
	\begin{equation} \label{convergenze}
	\begin{split}
		& \sup_{p\in\R^3} |f_n(p)-p|\le C(k_0)r_n \xrightarrow[n\to\infty]{}0,\\
		&\sup_{p\in\R^3} \bigg\|d(f_{n})_p - id\big|_{\R^3} \bigg\| \xrightarrow[n \to\infty]{}0,\\
		&\sup_{p\in\R^3} \| d^2 f_n (p)\| \equiv \sup_{p\in\R^3} \bigg(\sum_{i,j,k} | \pa_i\pa_j (f^k_n)(p) |^2\bigg)^{1/2}   \xrightarrow[n \to\infty]{}0.\\
	\end{split}
	\end{equation}
	Then $\|f_n-id|_{\R^3}\|_{C^2(\R^3)}\xrightarrow[n]{}0$. Hence we can write
	\begin{equation} \label{eta}
	\forall T>0\,\, \exists |\eta_n|\searrow0 : \qquad \cW(\Si_n)\ge \cW(\tilde{\Si}_n\cap B_T(0))=\cW(\Si'_n\cap B_T(0))+\eta_n.
	\end{equation}
	\noindent From now on call $I'_n:=\frac{1}{r_n}(I_{r_n},0,0)\con\pa\Si'_n$, $\ga'_n:=f_n\big(\frac{1}{r_n} (\pa B_{r_n}^{\R^2}(0)\cap D)\big)$ and $u'_n:\Om'_n:=f_n\big(\frac{1}{r_n}(\Om_n)\big)\to\R$ the functions parametrizing $\Si'_n$ corresponding to the functions $u_n$ in \eqref{inizio'}, that is
	\begin{equation*}
	u'_n((x,y))=\frac{1}{r_n}u_n\bigg( r_n(x,y)+(0,\Ga_{r_n}(r_nx))  \bigg)=\frac{1}{r_n}u_n\bigg( r_nx,r_ny+\Ga_{r_n}(r_nx)  \bigg).
	\end{equation*}
	By \eqref{inizio'} we have
	\begin{equation} \label{ref35}
	\begin{split}
		&(i)\quad\ga'_n \con \Om'_n,\\
		& (ii)\quad\Si'_n\sm ( graph(u'_n|_{\ga'_n})\cup I'_n) \simeq D\sqcup \Si_n\sm\Ga,\\
		&(iii)\quad |u'_n|\le r_n \quad\mbox{ on 
		} \ga'_n ,
	\end{split}
\end{equation}
\begin{equation}\label{ref35'}
\begin{split}
		&|\pa_i u'_n|(p)=\bigg|\de_i^x\pa_1 u_n\big|_{(r_nx,r_ny+\Ga_{r_n}(r_nx))} + (\de^y_i +\de^x_i \Ga'_{r_n}\big|_{r_nx} ) \pa_2 u_n\big|_{(r_nx,r_ny+\Ga_{r_n}(r_nx))}\bigg|\le\\&\qquad\qquad \le  2\big(1+ \sup |\Ga'_{r_n}|\big)r_n^2\le C_1 r_n^2 \qquad\qquad\mbox{ on 
		} \ga'_n ,
	\end{split}
	\end{equation}
	and since
	\begin{equation*}
	\begin{split}
		 |\pa_j\pa_i u'_n|^2((x,y))&=\big| \de_i^x\big[ \de^x_j r_n (\pa_{11}u_n)\big|_{(r_nx,r_ny+\Ga_{r_n}(r_nx))}+ \\ & \qquad +r_n(\de^y_j+\de^x_j\Ga'_{r_n}\big|_{r_nx})(\pa_{21}u_n)\big|_{(r_nx,r_ny+\Ga_{r_n}(r_nx))} \big]+\\&\qquad+r_n(\de^x_i\de^x_j\Ga_{r_n}''\big|_{r_nx})(\pa_2u_n) + (\de^y_i +\de^x_i \Ga'_{r_n}\big|_{r_nx} )\big[ \de^x_jr_n(\pa_{12}u_n)\big|_{(r_nx,r_ny+\Ga_{r_n}(r_nx))} +\\&\qquad+ r_n(\de^y_i +\de^x_i \Ga'_{r_n}\big|_{r_nx} )(\pa_{22}u_n)\big|_{(r_nx,r_ny+\Ga_{r_n}(r_nx))}  \big] \big|^2 \le \\
		&\le 2(k_0C\ep^{1/6}r_n)^2 +4 r_n^2 (1+\sup |\Ga'_{r_n}|)^4|\nabla^2 u_n|^2 \big|_{(r_nx,r_ny+\Ga_{r_n}(r_nx))} ,
	\end{split}
	\end{equation*}
	then
\begin{equation}\label{ref35''}
\begin{split}
		&\int_{\ga'_n } |\nabla^2 u'_n|^2 ((x,y))\,d\cH^1((x,y))\le \\ &\le 4\int_{\pa B^{\R^2}_{r_n}(0)\cap D } 2(k_0C\ep^{1/6}r_n)^2 +4 r_n^2 (1+\sup |\Ga'_{r_n}|)^4|\nabla^2 u_n|^2(\bar{x},\bar{y}) Jac_{\pa B^{\R^2}_{r_n}(0)}(f_n\circ D_{\frac{1}{r_n}})\, d\cH^1((\bar{x},\bar{y}))\le\\
		&\le 4\int_{\pa B^{\R^2}_{r_n}(0)\cap D } 2(k_0C\ep^{1/6}r_n)^2 +4 r_n^2 (1+\sup |\Ga'_{r_n}|)^4|\nabla^2 u_n|^2(\bar{x},\bar{y}) \frac{\bar{C}}{r_n}\, d\cH^1((\bar{x},\bar{y}))\le\\
		&\le 16\bar{C}\bigg( k_0^2 C^2\ep^{1/3}\frac{\pi}{2}r_n^2 +(1+\sup |\Ga'_{r_n}|)^4 \frac{r_n}{n}  \bigg) \le C_2r_n^2 + C_3 \frac{r_n}{n}.
	\end{split}
	\end{equation}
	
	\noindent Note that by construction $\ga'_n\xrightarrow[n]{}\pa B^{\R^2}_1(0)\cap \{y\ge 0 \}$ in $C^2$ norm. Then there exist functions $g^1_n:[0,1/4]\to\R$ such that $\ga'_n\cap \{y\le 1/4, x>0\}=\{(g^1_n(y),y)\}$ and $\| g^1_n- g^1\|_{C^2([0,1/4])}\searrow0$, with $g^1(y)=\sqrt{1-y^2}$; also extending $g^1_n$ and $g^1$ on $\R$ as we did in \eqref{ref34}, we can say that $\| g^1_n- g^1\|_{C^2(\R)}\searrow0$. Let $\chi^1\in C^\infty_c( B^{\R^2}_{1/4}((1,0)))$ be a cut off function such that $\chi|_{B^{\R^2}_{1/8}((1,0))}=1$ and consider the function
	\begin{equation*}
	F^1_n:\R^3\to \R^3 \qquad F^1_n((x,y,z))=\begin{cases}
		(x,y,z) & x\le\frac{1}{8},\\
		(1-\chi^1)\cdot (x,y,z)+\chi^1 \cdot( x-g^1_n(y)+g^1(y),y,z ) & x>\frac{1}{8}.
	\end{cases}
	\end{equation*}
	Performing the same construction around the point $(-1,0)$ and getting the corresponding function
	\begin{equation*}
	F^2_n:\R^3\to\R^3 \qquad F^2_n((x,y,z))=\begin{cases}
		(x,y,z) & x\ge-\frac{1}{8},\\
		(1-\chi^2)\cdot (x,y,z)+\chi^2 \cdot( x-g^2_n(y)+g^2(y),y,z ) & x<-\frac{1}{8},
	\end{cases}
	\end{equation*}
	we have that the diffeomorphism
	\begin{equation*}
	F_n:\R^3\to\R^3 \qquad F_n((x,y,z))= \begin{cases}
		F^1_n((x,y,z)) & x>\frac{1}{8},\\
		(x,y,z) & -\frac{1}{8}\le x\le \frac{1}{8},\\
		F^2_n((x,y,z)) & x<-\frac{1}{8},
	\end{cases}
	\end{equation*}
	is such that $\|F_n-id|_{\R^3}\|_{C^2(\R^3)}\to 0$ and $\ga''_n:=F_n(\ga'_n)$ intersects orthogonally the axis $\{y=0\}$. By the same arguments leading to \eqref{eta} we have $|\cW(\Si'_n\cap B_T(0))-\cW(F_n(\Si'_n)\cap B_T(0))|<\tilde{\eta}_n$ for some $|\tilde{\eta}_n|\searrow0$. Also by the same calculations in the \eqref{ref35}, \eqref{ref35'}, \eqref{ref35''} we have that the functions parametrizing $F_n(\Si'_n)$ satisfy the analogous relations.\\
	Hence, up to applying the sequence of diffeomorphisms $F_n$ on $\Si'_n$, in the following we can assume that
	\begin{equation} \label{ref36}
	\begin{split}
		& \eqref{eta} \quad\mbox{holds},\\
		& \eqref{ref35}, \eqref{ref35'}, \eqref{ref35''} \quad\mbox{hold},\\
		& \ga'_n \quad \mbox{ intersects orthogonally the axis } \{y=0\}.
	\end{split}
	\end{equation}\\

	\noindent Then by \eqref{ref36} the following extended functions are well defined.
	\begin{equation*}
	\begin{split}
		& \hat{u}_n': \Om'_n\cup\{ (x,y)|(x,-y)\in \Om'_n \} \to \R, \quad \hat{u}_n'((x,y)):=\begin{cases}
			u_n'((x,y)) & y\ge0,\\
			-u_n'((x,-y)) & y<0,
		\end{cases}\\
	\end{split}
	\end{equation*}
	and call $\hat{\ga}_n':=\ga_n'\cup \{ (x,y)|(x,-y)\in \ga'_n \}$. Note that by construction, the functions $\hat{u}_n'$ satisfy the inequalities in \eqref{ref35}, \eqref{ref35'}, \eqref{ref35''} (the notation adapted with $dom(\hat{u}_n')$ and $\hat{\ga}_n'$ in the right places).\\
	Consider now $R>>1$ and let $B_n=B^{\R^2}_R(0)\sm encl(\hat{\ga}_n')$.\\
	Let us adopt the following notation on trace operators
	\begin{equation*}
	\begin{split}
		&tr_{1,n}:W^{1,2}(B)\to L^2(\hat{\ga}_n'),\\
		&tr_2:W^{1,2}(B)\to L^2(\pa B_R^{\R^2}(0)),\\
		&tr:W^{1,2}(B)\to L^2 (\{y=0\}\cap B_n).
	\end{split}
	\end{equation*}
	A function $f\in C^1(\bar{B_n})\cap \{ tr_2(f)=0 \}$ verifies a Poincar\'{e}-like inequality as follows. If $T(\ro,\te)=\ro(\cos \te,\sen \te)$, then
	\begin{equation*}
	\begin{split}
		|f(x,y)|^2=|f\circ T(\bar{\ro},\te)|^2=\bigg| \int_{\bar{\ro}}^{R} \frac{\pa (f\circ T)}{\pa \ro}(\ro,\te)\,d\ro \bigg|^2\le (R-\bar{\ro}) \int_{\bar{\ro}}^{R} |\nabla f|^2|_{T(\ro,\te)}\, d\ro.
	\end{split}
	\end{equation*}
	Hence
	\begin{equation*}
	\begin{split}
		\|f\|^2_{L^2(B_n)} &= \int_0^{2\pi} \int_{r(\te)}^R  |f|^2|_{T(\bar{\ro},\te)} \bar{\ro} \,  d\bar{\ro} d\te \le \int_0^{2\pi} \int_{r(\te)}^R  (R-\bar{\ro})\int_{\bar{\ro}}^R |\nabla f|^2|_{T(\ro,\te)} \,d\ro\, \bar{\ro} \, d\bar{\ro}d\te \le\\
		&\le \int_0^{2\pi} \int_{r(\te)}^R \bigg(R-\frac{1}{2}\bigg)  \int_{\bar{\ro}}^R |\nabla f|^2|_{T(\ro,\te)}\ro \,d\ro \,  d\bar{\ro}d\te \le\\ &\le 
		\bigg(R-\frac{1}{2}\bigg) \int_0^{2\pi} \int_{r(\te)}^R   \int_{r(\te)}^R |\nabla f|^2|_{T(\ro,\te)}\ro \,d\ro \,  d\bar{\ro}d\te \le \\&\le \bigg(R-\frac{1}{2}\bigg)^2 \|\nabla f\|^2_{L^2(B_n)},
	\end{split}
	\end{equation*}
	where we used $\bar{\ro}\le \ro$ in the second inequality, $\bar{\ro}\ge r(\te)$ in the third inequality and we assumed $n$ sufficiently large so that $r(\te)>1/2$ for any $\te$.\\
	Similarly if $\nabla f\in C^1(\bar{B_n})\cap \{ tr_2(\nabla f)=0 \}$, then $\|\nabla f\|^2_{L^2(B_n)}\le 2(R-1/2)^2 \|\nabla^2 f \|^2_{L^2(B_n)}$.\\
	By approximation we get that if $v\in V=W^{2,2}\cap \{ tr_2(v)=0, tr_2(\nabla v)=0 \}$, then
	\begin{equation} \label{poincare'}
	\begin{split}
		& \|v\|_{L^2(B_n)}\le (R-1/2)\|\nabla v\|_{L^2(B_n)},\\
		& \|\nabla v\|_{L^2(B_n)}\le \sqrt{2}(R-1/2)\|\nabla^2 v\|_{L^2(B_n)}.
	\end{split}
	\end{equation}
	Hence by direct methods, \eqref{poincare'} and continuity of trace operators, for any $n$ there exists a solution to the following minimization problem
	\begin{equation*}
	\begin{split}
		\min \bigg\{\int_{B_n} |\nabla^2 v|^2 	\quad|\quad v\in W^{2,2}(B_n):\,\, & tr_{1,n}(v)=\hat{ u}_n',\,\, tr_{1,n}(\nabla v)=\nabla\hat{ u}_n',\\
		& tr_2(v)=0,\,\, tr_2(\nabla v)=0,\\
		& tr(v)=0  \bigg\}.
	\end{split}
	\end{equation*}
	Call such minimizer $w_n$. Hence $w_n$ satisfies in a weak sense the equation
	\begin{equation*}
	\begin{cases}
		\De^2w_n=0 & \mbox{on } B_n,\\
		tr_{1,n}(w_n)=\hat{ u}_n',\\ tr_{1,n}(\nabla w_n)=\nabla\hat{ u}_n',\\
		tr_2(w_n)=0,\\ tr_2(\nabla w_n)=0,\\
		tr(w_n)=0 ,
	\end{cases}
	\end{equation*}
	that implies $w_n\in C^\infty(\bar{B_n})$.\\
	By \eqref{ref36} and the same calculations of the proof of Lemma \ref{xiinterno} we get
	\begin{equation*}
	\int_{B_n} |\nabla^2 w_n|^2 \le C(B_n) \|\nabla^2 \hat{ u}'_n\|^2_{L^2(\hat{\ga}'_n)}\le C(B_n) r_n (C_2r_n+C_3 1/n),
	\end{equation*}
	and since $B_n\to B^{\R^2}_R(0)\sm B^{\R^2}_1(0)$ in $C^2$ norm, then $C(B_n)\le C^*$ for some $C^*$ independent of $n$. Hence we get
	\begin{equation*}
	D(graph(w_n)),\cW(graph(w_n))\le \ep_n \xrightarrow[n]{}0.
	\end{equation*}
	\noindent Extend now $w_n$ to the value $0$ outside of $B_R^{\R^2}(0)$ to get a function $w_n\in C^{1,1}(enlc(\hat{\ga}'_n)^c)$. We consider the following $C^{1,1}$ composite surface
	\begin{equation*}
	S_n =[ (\Si'_n)^*\cup cl(graph(w_n))]\sm \{ (x,y,z)|y<0 \},
	\end{equation*}
	where $(\Si'_n)^*$ denotes the obvious truncation at $\ga'_n$.\\
	Choosing $T>R$ in \eqref{ref36} we have
	\begin{equation*} 
	\cW(\Si_n)\ge \cW(S_n) -\ep_n+\eta_n.
	\end{equation*}
	Finally we observe that by construction $S_n$ is a $C^{1,1}$ surface of genus $1$ with the axis $\{y=0\}$ as boundary. Then by approximation by Theorem \ref{thmr} we conclude that
	\begin{equation*}
	\liminf_n \cW(\Si_n) \ge \be_1-4\pi,
	\end{equation*}
	which concludes the proof.
\end{proof}

\begin{proof}[Proof of Theorem \ref{thmmain}] Point $i)$ follows from Lemma \ref{xiinterno} and Lemma \ref{xibordo}. For point $ii)$ let $T\con\R^3$ be the standard Willmore minimizing torus, that is the stereographic projection in $\R^3$ of the Clifford torus $\S^1\times \S^1\con S^3$, and fix a point $q\in T$ such that the Gaussian curvature $K$ satisfies $K(q)>0$. Then perform the spherical inversion $I_{1,q}(T\sm\{q\})$. After isometry we get an asymptotically flat torus $S$ with asymptotic plane $\{z=0\}$ and for $\eta>0$ small enough we can assume that $S\cap \{z<\eta\}$ identifies an end of the surface. Removing such end we get a surface $S_\eta$ with planar convex boundary $\Ga_\eta$ and energy $\cW(S_\eta)<2\pi^2-4\pi$, then by $i)$, if $\inf_{\cC(\Si_\cg)}\cW= \inf_{\cC_{imm}(\Si_\cg)} \cW$ for such $\Ga_\eta$, then problem \eqref{problemmin} with boundary $\Ga_\eta$ has minimizers.\\
Finally point $iii)$ follows from the fact that the arguments in the proof of Lemma \ref{xiinterno} and Lemma \ref{xibordo} are local. More precisely, let $\cg\ge2$ and assume the hypotheses in $iii)$. By Proposition \ref{simon} a minimizing sequence $\Si_n$ still converge to $D$ as varifolds and by \eqref{diffeom} we get the existence of $P\ge1$ bad points $\xi_1,...,\xi_P$ each absorbing a quantum $\cg_i\ge1$ of genus with $\sum_{i=1}^P \cg_i =\cg$, in the sense that
	\begin{equation*}
	\exists\si_n\searrow 0 :\quad D\sm\bigcup_{i=1}^P B_{\si_n}(\xi_i) \simeq \Si_n \sm  \bigcup_{i=1}^P B_{\si_n}(\xi_i).
	\end{equation*}
	Hence applying the same arguments in the proof of Lemma \ref{xiinterno} and Lemma \ref{xibordo} at each point $\xi_i$ we now get the inequality
	\begin{equation*}
	\cW(\Si_n)\ge \sum_{i=1}^P (\be_{\cg_i}-4\pi)=\sum_{i=1}^P e_{\cg_i}.
	\end{equation*}
	Then the thesis follows by the inequality $ \sum_{i=1}^P e_{\cg_i} > e_\cg$ proved in \cite{BaKu}.
\end{proof}

\textcolor{white}{text}

\appendix

\section{Curvature varifolds with boundary}

\noindent In this appendix we recall the definitions and the results about curvature varifolds with boundary that we need throughout the whole work. This section is based on \cite{Ma} (see also \cite{SiGMT}, \cite{Hu}).\\

\noindent Let $\Om\con\R^k$ be an open set, and let $1<n\le k$. We identify a $n$-dimensional vector subspace $P$ of $\R^k$ with the $k\times k$-matrix $\{P_{ij}\}$ associated to the orthogonal projection over the subspace $P$. Hence the Grassmannian $G_{n,k}$ of $n$-spaces in $\R^k$ is endowed with the Frobenius metric of the corresponding projection matrices. Moreover given a subset $A\con\R^k$, we define $G_n(A)=A\times G_{n,k}$, endowed with the product topology. A general $n$-varifold $V$ in an open set $\Om\con\R^k$ is a non-negative Radon measure on $G_n(\Om)$. The varifold convergence is the weak* convergence of Radon measures on $G_n(\Om)$, defined by duality with $C^0_c(G_n(\Om))$ functions.\\
We denote by $\pi:G_n(\Om)\to\Om$ the natural projection, and by $\mu_V=\pi_\sharp(V)$ the push forward of a varifold $V$ onto $\Om$. The measure $\mu_V$ is called induced (weight) measure in $\Om$.\\
Given a couple $(M,\te)$ where $M\con\Om$ is countably $n$-rectifiable and $\te:M\to\N_{\ge1}$ is $\cH^n$-measurable, the symbol $\bv(M,\te)$ defines the (integer) rectifiable varifold given by
\begin{equation*}
	\int_{G_n(\Om)} \vp(x,P)\,d\bv(M,\te)(x,P) = \int_M \vp(x,T_xM)\,\te(x)\,d\cH^n(x),
\end{equation*}
where $T_xM$ is the generalized tangent space of $M$ at $x$ (which exists $\cH^n$-ae since $M$ is rectifiable). The function $\te$ is called density or multiplicity of $\bv(M,\te)$. Note that $\mu_V=\te\cH^n\res M$ in such a case.\\

\noindent From now on we will always understand that a varifold $V$ is an integer rectifiable one.\\

\noindent We say that a function $\vec{H}\in L^1_{loc}(\mu_V;\R^k)$ is the generalized mean curvature of $V=\bv(M,\te)$ and $\si_V$ Radon $\R^k$-valued measure on $\Om$ is its generalized boundary if
\begin{equation*}
	\int \div_{TM} X \, d\mu_V = - n \int \lgl \vec{H}, X \rgl \,d\mu_V + \int X\,d\si_V,
\end{equation*}
for any $X\in C^1_c(\Om;\R^k)$, where $\div_{TM} X$ is the $\cH^n$-ae defined tangential divergence of $X$ on the tangent space of $M$.\\
If $V$ has generalized mean curvature $\vec{H}$, the Willmore energy of $V$ is defined to be
\begin{equation*}
	\cW(V)=\int |H|^2\,d\mu_V.
\end{equation*}
The operator $X\mapsto\de V(X):=\int\div_{TM} X\,d\mu_V$ is called first variation of $V$. Observe that for any $X\in C^1_c(\Om;\R^k)$, the function $\vp(x,P):=\div_{P}(X)(x)=tr(P\nabla X(x))$ is continuous on $G_n(\Om)$. Hence, if $V_n\to V$ in the sense of varifolds, then $\de V_n(X)\to\de V(X)$.\\

\noindent By analogy with integration formulas classically known in the context of submanifolds, we say that a varifold $V=\bv(M,\te)$ is a curvature $n$-varifold with boundary in $\Om$ if there exist functions $A_{ijk}\in L^1_{loc}(V)$ and a Radon $\R^k$-valued measure $\pa V$ on $G_n(\Om)$ such that
\begin{equation*}
\begin{split}
	\int_{G_n(\Om)} P_{ij}\pa_{x_j}\vp(x,P) &+ A_{ijk}(x,P)\pa_{P_{jk}}\vp(x,P)   \,dV(x,P) =\\&=  n\int_{G_n(\Om)}\vp(x,P) A_{jij}(x,P) \,dV(x,P) + \int_{G_n(\Om)} \vp(x,P)\,d\pa V_i(x,P),
\end{split}
\end{equation*}
for any $i=1,...,k$ for any $\vp\in C^0_c(G_n(\Om))$. The rough idea is that the term on the left is the integral of a tangential divergence, while on the right we have integration against a mean curvature plus a boundary term. The measure $\pa V$ is called boundary measure of $V$.

\begin{thm}[\cite{Ma}]
	Let $V=\bv(M,\te)$ be a curvature varifold with boundary on $\Om$. Then the following hold true.\\
	i) $A_{ijk}=A_{ikj}$, $A_{ijj}=0$, and $A_{ijk}=P_{jr}A_{irk}+P_{rk}A_{ijr}=P_{jr}A_{ikr}+P_{kr}A_{ijr}$.\\
	ii) $P_{il}\pa V_l(x,P)=\pa V_i(x,P)$ as measures on $G_n(\Om)$.\\
	iii) $P_{il}A_{ljk}=A_{ijk}$.\\
	iv) $H_i(x,P):=\frac{1}{n}A_{jij}(x,P)$ satisfies that $P_{il}H_l(x,P)=0$ for $V$-ae $(x,P)\in G_n(\Om)$.\\
	v) $V$ has generalized mean curvature $\vec{H}$ with components $H_i(x,T_xM)$ and generalized boundary $\si_V=\pi_\sharp(\pa V)$.
\end{thm}

\noindent We call the functions $\sff_{ij}^k(x):=P_{il}A_{jkl}$ components of the generalized second fundamental form of a curvature varifold $V$. Observe that $\sff_{jj}^k=P_{jl}A_{jlk}=A_{jjk}-P_{kl}A_{jjl}=A_{jkj}-P_{kl}A_{jlj}=nH_k-nP_{kl}H_l=nH_k$, and $A_{ijk}=\sff^k_{ij}+\sff^j_{ki}$.\\

\noindent In conclusion we state the compactness theorem that we use in this work.

\begin{thm}[\cite{Ma}]\label{thm1}
	Let $p>1$ and $V_l$ a sequence of curvature varifolds with boundary in $\Om$. Call $A_{ijk}^{(l)}$ the functions $A_{ijk}$ of $V_l$. Suppose that $A_{ijk}^{(l)}\in L^p(V)$ and
	\begin{equation*}
		\sup_l \quad\bigg\{\mu_{V_l}(W) + \int_{G_n(W)} \bigg|\sum_{i,j,k} |A^{(l)}_{ijk}|\bigg|^p\,dV_l + |\pa V_l|(G_n(W))\bigg\}\quad\le C(W)<+\infty
	\end{equation*}
	for any $W\con\con G_n(\Om)$, where $|\pa V_l|$ is the total variation measure of $\pa V_l$. Then:\\
	i) up to subsequence $V_l$ converges to a curvature varifold with boundary $V$ in the sense of varifolds. Moreover $A^{(l)}_{ijk} V_l\to A_{ijk}V$ and $\pa V_l\to \pa V$ weakly* as measures on $G_n(\Om)$;\\
	ii) for every lower semicontinuous function $f:\R^{k^3}\to[0,+\infty]$ it holds that
	\begin{equation*}
		\int_{G_n(\Om)} f(A_{ijk})\,dV \le \liminf_l \int_{G_n(\Om)} f(A_{ijk}^{(l)})\,dV_l.
	\end{equation*}
\end{thm}

\noindent It follows from the above theorem that the Willmore energy is lower semicontinuous with respect to varifold convergence of curvature varifolds with boundary satisfying the hypotheses of Theorem \ref{thm1}.\\

\noindent Finally we give the definition of image varifold used in this work. Let $\bv(M,\te)$ be an integer rectifiable varifold in $\Om$, and let $f:A\con\Om\to B\con\R^k$ be a Lipschitz proper function with $M\con A$. Then the image varifold $\Imm(f)$ is defined by $\Imm(f)=\bv(f(M),\tilde{\te})$ with $\tilde{\te}(y)=\displaystyle\sum_{x\in f^{-1}(x)\cap M}\te(x)$.\\

\textcolor{white}{text}
\section{Useful results}

\noindent For the convenience of the reader, here we collect some useful technical results that we need in the proofs of this work.

\begin{lemma}[Graphical Decomposition]
	Let $\Si\con\R^3$ be a smooth immersed surface with embedded boundary $\Ga$, and let $p_0\in\Si$. For any $\be>0$ there is $\ep_0=\ep_0(\be)>0$ such that if $\ep\in(0,\ep_0]$, if $|\Si\cap \overline{B_\ro(p_0)}|\le\be\ro^2$, and if $D(B_\ro(p_0))\le\ep^2$, then the following holds.\\
	There are pairwise disjoint closed topological discs $P_1,...,P_N\con\Si$ with $\sum_{j=1}^N \diam P_j \le C\sqrt{\ep}\ro$ such that
	\begin{equation*}
		\Si\cap B_{\frac{\ro}{2}}(p_0)\sm\bigg(\bcup_{i=1}^N P_j\bigg)=\bigg(\bcup_{i=1}^M graph(u_i)\bigg)\cap B_{\frac{\ro}{2}}(p_0),
	\end{equation*}
	where $u_i\in C^\infty(\overline{\Om_i};L_i^\perp)$ with $L_i$ plane in $\R^3$ and $\Om_i$ piecewise smooth bounded connected domain in $L_i$ of the form $\Om^0_i\sm\big(\cup_k d_{i,k}\big)$, where $\Om^0_i$ is simply connected and $d_{i,k}$ are pairwise disjoint closed discs in $L_i$ such that
	\begin{equation*}
		M\le C\be, \qquad \sup_{\Om_i} \frac{|u_i|}{\ro}+\sup_{\Om_i} |\nabla u_i| \le C\ep^{1/6}.
	\end{equation*}
	Moreover, for any $\si\in(\ro/4,\ro/2)$ such that $\Si$ intersects $\pa B_\si(p_0)$ transversely and $\pa B_\si(p_0)\cap (\cup_j P_j)=\epty$ we have
	\begin{equation*}
		\Si\cap \overline{B_\si(p_0)}=\bcup_{i=1}^M D_{\si,i},
	\end{equation*}
	where $D_{\si,i}$ is a topological disc with $graph (u_i)\cap \overline{B_\si(p_0)}\con D_{\si,i}$, and $D_{\si,i}\sm graph u_i$ is a union of a subcollection of the $P_j$'s.\\
	Moreover, if $p_0\not\in\Ga$, then for $\ro<d(p_0,\Ga)$ the domains $\Om_i$ are smooth and $d_{i,k}\cap \Ga=\epty$.
\end{lemma}

\begin{proof}
	The proof immediately follows by Lemma 2.1 in \cite{SiEX} and the observations in \cite{ScBP} at p. 280.
\end{proof}

\begin{lemma}[Biharmonic Comparison, Lemma 2.2 in \cite{SiEX}]\label{lem1}
	Let $u\in C^2(U)$ with $U\con\R^2$ smooth open neighborhood of $\pa B^{\R^2}_\ro(0)$. Call $\Si$ the surface parametrized by $u$ as graph. Suppose that $|\nabla u|\le \ep$. Let $w\in C^\infty(\overline{B^{\R^2}_\ro(0)})$ such that
	\begin{equation*}
	\begin{cases}
		\De^2 w=0 & B^{\R^2}_\ro(0),\\
		w=u 		&\pa B^{\R^2}_\ro(0),\\
		\nabla w= \nabla u   &\pa B^{\R^2}_\ro(0).
	\end{cases}
	\end{equation*}
	Then
	\begin{equation*}
		\int_{B^{\R^2}_\ro(0)} |\nabla^2 w|^2 \le C\ro \int_{graph\big(u|_{\pa B^{\R^2}_\ro(0)}\big)} |\sff_\Si|^2\, d\cH^1,
	\end{equation*}
	where $C=C(\ep)$ depends only on $\ep$ and it is uniformly bounded for $\ep\in[0,1]$.
\end{lemma}

\begin{lemma}[Boundary Biharmonic Comparison]\label{lem2}
	Let $\psi:(-a,a)\to\R$ with $\psi\in C^2$ parametrize as graph a curve $\Ga\con\R^2$. Assume $\psi(0)=0$, $|\psi'|<\de\in(0,1)$, and $|\psi''|\le\La$. Define $A:=\{(x,y)\,|\,y\ge\psi(x)\}$ and let $\ro\in(-a/2,a/2)$. Let $u\in C^2(U)$ with $U$ open neighborhood in $A$ of $\pa B^{\R^2}_\ro(0)\cap A$. Call $\Si$ the surface with boundary parametrized by $u$ as graph. Suppose
	\begin{equation*}
		u|_{\Ga\cap U}=0, \qquad |\nabla u|\le \ep.
	\end{equation*} 
	Let $w\in C^\infty(\overline{B^{\R^2}_\ro(0)})$ such that
	\begin{equation*}
	\begin{cases}
		\De^2 w=0 & B^{\R^2}_\ro(0),\\
		w=\bar{u} 		&\pa B^{\R^2}_\ro(0),\\
		\nabla w= \nabla \bar{u}   &\pa B^{\R^2}_\ro(0),\\
		w=0 & \Ga\cap B^{\R^2}_\ro(0),
	\end{cases}
	\end{equation*}
	where $\bar{u}, \nabla\bar{u}$ denote continuous extensions of $u,\nabla u$ on $\pa B^{\R^2}_\ro(0)$ with $\|\bar{u}\|_\infty\le\|u\|_\infty, \|\nabla\bar{u}\|_\infty\le C(\La)\|\nabla u\|_\infty$.
	Then
	\begin{equation*}
		\int_{B^{\R^2}_\ro(0)} |\nabla^2 w|^2 \le C\ro \int_{graph\big(u|_{\pa B^{\R^2}_\ro(0)\cap A}\big)} |\sff_\Si|^2\, d\cH^1,
	\end{equation*}
	where $C=C(\ep,\La)$ depends only on $\ep,\La$ and it is uniformly bounded for $\ep\in[0,1]$ and $\La\in[0,\La_0]$ for a given $\La_0<+\infty$.
\end{lemma}

\begin{proof}
	Up to $C^2$-diffeomorphism we can assume that $\psi\equiv0$. In this case the surface $\Si$ can be extended to a surface $\bar{\Si}$ inside $B_{\frac{7}{4}\ro}$ by an odd reflection about the segment $\Ga$, so that $\bar{\Si}\cap B_{\frac{5}{4}\ro}$ is without boundary. Extending correspondingly the function $u$, the proof follows applying the very same estimates used in the proof of Lemma \ref{lem1}.
\end{proof}

\begin{remark}
	The functions $w$ in Lemma \ref{lem1} and in Lemma \ref{lem2} can be chosen to be a minimizers of $\int_{B^{\R^2}_\ro(0)}|\nabla^2 v|^2$ or of $\int_{B^{\R^2}_\ro(0)} |\De v|^2$, among competitors $v$ satisfying the same boundary conditions.
\end{remark}

\begin{lemma}[Poincar\'{e}-type inequality]
	Let $\de\in(0,1/2)$. Let $\Om\con B^{\R^2}_1(0)=:B$ a domain of the form $B\sm E$ such that: $\cL^1(\pi_x(E))\le 1/2$ and $\cL^1(\pi_y(E))< \de$, where $\pi_i$ is the projection on the $i$-axis. Then there exists a universal constant $C$ such that for any $f\in W^{1,2}(\Om)$ and any subdomain $\om\con\Om$ it holds that
	\begin{equation*}
		\inf_{\la\in\R} \int_\om |f-\la|^2\,dx \le C \int_\om |\nabla f|^2\,dx + C\de \sup_\om |f|^2.
	\end{equation*}
\end{lemma}

\begin{proof}
	For a subdomain $\om\con\Om$ and $f\in W^{1,2}(\Om)$, the restriction $f|_\om$ is a Sobolev function in $W^{1,2}(\om)$. Hence the proof follows by applying Lemma A.1 in \cite{SiEX}.
\end{proof}

\vspace{0.5cm}
\textcolor{white}{text}\\
\noindent\emph{Acknowledgments.} I warmly thank Matteo Novaga for his useful suggestions and comments.\\


\end{document}